\documentclass{amsart}
\usepackage{tikz}
\usetikzlibrary{shapes.geometric}
\usepackage[center]{caption}
\usepackage{xcolor}
\usepackage{amssymb,latexsym,amsmath,extarrows}
\usepackage{graphicx,mathrsfs,comment}
\usepackage{hyperref,url}
\usepackage{pict2e}

\usepackage{amstext}
\usepackage{bbm}

\numberwithin{equation}{section}

\usepackage{esint}

\newtheorem{theorem}{Theorem}[section]
\newtheorem{lemma}[theorem]{Lemma}

\newtheorem{remark}[theorem]{Remark}

\newtheorem{definition}[theorem]{Definition}
\newtheorem{corollary}[theorem]{Corollary}

\newtheorem{conjecture}[theorem]{Conjecture}

\makeatletter
\newcommand{\barredsum}{%
  \DOTSB\mathop{\mathpalette\@barredsum\relax}\slimits@
}
\newcommand{\@barredsum}[2]{%
  \begingroup
  \sbox\z@{$#1\sum$}%
  \setlength{\unitlength}{\dimexpr2pt+\ht\z@+\dp\z@\relax}%
  \@barredsumthickness{#1}%
  \vphantom{\@barredsumbar}%
  \ooalign{$\m@th#1\sum$\cr\hidewidth$#1\@barredsumbar$\hidewidth\cr}%
  \endgroup
}
\newcommand{\@barredsumbar}{%
  \vcenter{\hbox{\begin{picture}(0,1)\roundcap\Line(0,0)(0,1)\end{picture}}}%
}
\newcommand{\@barredsumthickness}[1]{
  \linethickness{%
    1.25\fontdimen8
      \ifx#1\displaystyle\textfont\else
      \ifx#1\textstyle\textfont\else
      \ifx#1\scriptstyle\scriptfont\else
      \scriptscriptfont\fi\fi\fi 3
  }%
}
\makeatother

\newcommand{\al}{\alpha}
\newcommand{\be}{\beta}

\newcommand{\Ga}{\Gamma}
\newcommand{\de}{\delta}

\newcommand{\e}{\varepsilon}

\newcommand{\ka}{\kappa}
\newcommand{\la}{\lambda}

\newcommand{\si}{\sigma}
\newcommand{\Si}{\Sigma}

\newcommand{\om}{\omega}
\newcommand{\Om}{\Omega}

\newcommand{\cs}{\mathcal S}

\newcommand{\cu}{\mathcal U}
\newcommand{\cq}{\mathcal Q}
\newcommand{\cp}{\mathcal P}

\newcommand{\co}{\mathcal O}

\newcommand{\cb}{\mathcal B}

\newcommand{\wt}{\widetilde}
\newcommand{\wh}{\widehat}

\newcommand{\ZR}{\mathbb{R}}
\newcommand{\ZT}{\mathbb{T}}
\newcommand{\ZZ}{\mathbb{Z}}

\newcommand{\ZW}{\mathbb{W}}

\newcommand{\ZS}{\mathbb{S}}

\newcommand{\Id}{{\bf{1}}}

\newcommand{\ba}{{\bf a}}
\newcommand{\bd}{{\bf d}}

\newcommand{\bn}{{\bf n}}

\newcommand{\bx}{{\bf x}}

\newcommand{\cT}{{\mathcal T}}

\newcommand{\bq}{{\bf q}}

\newcommand{\BR}{{\rm{Br}}}

\newcommand{\T}{\mathbb{T}  }

\newcommand{\1}{\mathbbm{1}}

\newcommand{\dist}{{\rm dist}}

\newcommand{\BL}{\textup{BL}}

\newcommand{\ang}{\measuredangle}

\begin{document}

\title{An improved restriction estimate in $\ZR^3$}

\author{Hong Wang} \address{ Hong Wang\\  Department of Mathematics\\ University of California, Los Angeles, USA} \email{hongwang@math.ucla.edu}

\author{Shukun Wu} \address{ Shukun Wu\\  Department of Mathematics\\ California Institute of Technology, USA} \email{skwu@caltech.edu}

\date{}
\begin{abstract}
We improve the $L^{p}\rightarrow L^p$ restriction estimate in $\ZR^3$ to the range $p>3+3/14$, based on some Kakeya type incidence estimates and the refined decoupling theorem.
\end{abstract}

\maketitle

\section{Introduction}

Let $f(\xi)$ be a function supported in the two-dimensional unit ball $B^2(0,1)$. Denote by $Ef$ be the extension operator
\begin{equation}
    Ef(x',x_3)=\int_{\ZR^3} e^{ix'\cdot\xi}e^{x_3|\xi|^2}f(\xi)d\xi,
\end{equation}
where $x=(x',x_3)\in\ZR^3$. The following conjecture was made by Stein, which is known as (Stein's) restriction conjecture (for paraboloid).
\begin{conjecture}[Restriction conjecture in $\ZR^3$]
\label{restriction-conj}
Suppose $f\in L^p(B^2(0,1))$. Then for any $p>3$,
\begin{equation}
\label{restriction-esti}
    \|Ef\|_p\leq C_p\|f\|_p.
\end{equation}
\end{conjecture}

Conjecture \eqref{restriction-conj} was first proved for $p>4$ by Tomas and Stein. In 1991, by relating it to the Kakeya problem, Bourgain improved the range to $p>3+7/8$ in the milestone work  \cite{Bourgain-Besicovitch}.  In 2002, Tao \cite{Tao-bilinear} improved \eqref{restriction-esti} to  $p>3+1/3$ using a two-ends argument, which was introduced by Wolff \cite{Wolff-bilinear} a bit earlier. Then in 2010, Bourgain-Guth \cite{Bourgain-Guth-Oscillatory} further improved the range to $p>3+5/17$ via a  broad-narrow decomposition. More recently, Guth \cite{Guth-R3} brought ideas from incidence geometry, the polynomial partitioning, and improved \eqref{restriction-esti} to the range $p>3+1/4$. The best known result so far was due to the first author \cite{Wang-restriction-R3}, showing \eqref{restriction-esti} for $p>3+3/13$ by combining the polynomial partitioning technique and the two-ends argument, and exploring a so-called broom structure. 

We give another small improvement based on some Kakeya estimates and the refined decoupling theorem.

\begin{theorem}
\label{main-thm}
The restriction estimate \eqref{restriction-esti} is true for $p>3+3/14$.
\end{theorem}

Using an epsilon removal argument by Tao \cite{Tao-BR-restriction}, it suffices to prove that for any $p>3+3/14$ and   $\e>0$, there exists a constant $C_{\e}$ such that the following holds: for $R>1$,
\begin{equation}
\label{epsilon-removal}
    \|Ef\|_{L^p(B_R)}\leq C_\e R^\e\|f\|_p.
\end{equation}
For any $R^{-1/2}$-ball $\theta$ in $B^2(0,1)$,  let $f_\theta=f\Id_\theta$. In the paper we will prove 
\begin{equation}
\label{main-esti}
    \|Ef\|_{L^p(B_R)}\leq C_\e R^\e\|f\|_2^{2/p}\sup_{\theta}\|f_\theta\|_{L^2_{ave}}^{1-2/p},
\end{equation}
where $\|f_\theta\|_{L^2_{ave}}^2=|\theta|^{-1}\|f_\theta\|_2^2$. Note that \eqref{main-esti}  implies \eqref{epsilon-removal} by (restricted type) real interpolation, and hence Theorem \ref{main-thm}.

\smallskip

Bourgain \cite{Bourgain-Besicovitch} originates the idea of  studying the restriction problem using wave packets. A wave packet serves as a building block for $Ef$, and is essentially supported  in a long thin tube. To this end, it is important to understand how different wave packets interact with each other. An effective tool to understand the interaction is the polynomial partitioning introduced by Guth \cite{Guth-R3}. Broadly speaking, polynomial partitioning (iteration) allows one to partition the space into cells, so that in each cell $Ef$ contributes roughly the same, and there is a certain algebraic constraint among all the cells that controls their interaction via wave packets. In addition to this algebraic constraint, the oscillation of wave packets also plays an important role, which is usually handled by induction on scales.

The main idea of this paper is to apply a refined Wolff's hairbrush estimate (at two different scales) to bound the number of tubes through each cell when the cells are relatively concentrated and apply the refined decoupling theorem (\cite{GIOW} and independently by Du and Zhang) when the cells are spread-out.

Let us briefly explain the refined Wolff's hairbrush estimate. Let $\mathcal{T}$ be a set of $\delta$-tubes that has less than $m\geq 1$ tubes in each $\delta$-separated direction. A shading $Y: \mathcal{T}\rightarrow \mathbb{R}^3$ is a map such that $Y(T)\subset T$. Suppose $|Y(T)|/|T|\sim \lambda$ for each $T\in \mathcal{T}$. The refined Wolff's hairbrush estimate says that if the shading on each tube is ``two-ends'', then for a typical point in $\bigcup Y(T)$, the multiplicity is $\lessapprox \lambda^{-1}\delta^{-1/2}m$, improving upon Wolff's original bound $\lessapprox \lambda^{-3/2}\delta^{-1/2}m.$

Our proof is based on the framework built in \cite{Wang-restriction-R3}.  Morally speaking, after  polynomial partitioning, the cells are organized into a collection of fat surfaces $S$, which is the $r^{1/2}$--neighborhood of a (low degree) algebraic surface intersecting a ball of radius $r$, for some  $1\leq r \leq R.$

If $r\leq R^{1/2}$, then it is already proved in \cite{Wang-restriction-R3} that the restriction estimate holds for $p> 3.20$, which is better than what we currently prove. 

If $R^{1/2}<r<R^{2/3}$, then we observe that the wave packets of $Ef$ need to be ``sticky'', otherwise we can obtain improved ``broom estimate''.  Here ``sticky'' means that the wave packets 
pointing in nearby directions (within an $r^{-1/2}$-directional cap) are contained in a small number of parallel $R r^{-1/2}$-tubes. In this case, we apply the aforementioned refined hairbrush bound to the set of $R r^{-1/2}$-tubes and conclude that the number of those tubes through each $S$ is small, which is an improvement over the polynomial Wolff axiom. 

The case when $r \geq R^{2/3}$ is more involved. Roughly speaking, if the cells (who have a small diameter $R^{O(\delta)}$) are spread-out among many $R^{1/2}$-balls, then we apply refined decoupling estimate.  Otherwise, we apply the bilinear restriction estimate locally and then use a square function to control all local contributions. The square function will further be estimated using the refined hairbrush bound.

\medskip

\noindent {\bf Organization of the paper.} The paper is organized as the following: Section 2 contains several technical results, for instance, pruning of wave packets, a refined Wolff's hairbrush estimate. We also review the polynomial partitioning in this section. Section 3 and 4 contains some quantitative two-ends reductions. In Section 5 we deal with the case $r\leq R^{2/3}$ and the case $r\geq R^{2/3}$ is discussed in Section 6.

\medskip

\noindent
{\bf Acknowledgement.} We  would like to thank Larry Guth and Ruixiang Zhang for helpful discussions on Lemma \ref{polynomial-plane-lemma}. The second author would like to thank Xiaochun Li for discussions about Lemma \ref{planar-refinement-lem}, and would like to thank Nets Katz for discussions about Lemma \ref{two-ends-hairbrush}.

\medskip

\noindent{\bf Notations: }

\noindent$\bullet$ We normalize $\|f\|_2=1$ in the whole paper.

\noindent$\bullet$ We use $d(\theta')$ to denote the diameter of a small ball $\theta'$ in $\ZR^2$ (or a small cap $\theta'$ on the paraboloid). Sometimes when we say $\theta'$ is a ``scale-$r$" directional cap, we mean $d(\theta')=r^{-1/2}$. Denote by $\Theta[r]$ the collection of scale-$r$ directional caps. The letter $\theta$ is reserved for $R^{-1/2}$-balls (scale-$R$ directional caps).

\noindent$\bullet$ We write $A\lessapprox B$ if $A\leq C_\be R^{\be} B$ for all small $\be>0$. 

\noindent$\bullet$ We write $A=\Om(B)$ if $A\geq cB$ for some absolute constant $c>0$.

\noindent$\bullet$ Let $\text{Rel}\in\{=,\sim,\leq,\lesssim,\lessapprox\}$ be a binary relation. To avoid abundant notations that handle rapidly decreasing terms, if $a(R),b$ are two real numbers, then in this paper $a(R)~\text{Rel}~b$ may occasionally (when Schwartz tails appear) mean
\begin{equation}
    a(R)~\text{Rel}~(b+C_{\eta}R^{-\eta}) \;\; \text{for all $R \geq 1$ and $\eta>0$}.
\end{equation}

\noindent$\bullet$ For a set $\cq$ we define $\bigcup_\cq=\bigcup_{Q\in\cq}Q$.

\noindent$\bullet$ Here are some numerical factors we will use in the paper: $d$ is a large constant depending on $\e$, $\de=\e^{3}$, $\e_0=\e^{10}$, $\e_0'=\e^{15}$.

\section{Preliminary tools and lemmas}

\subsection{Wave packet and its pruning}
\label{wp-pruning-section}

The scale-$R$ wave packet decomposition is
\begin{equation}
\label{wp-decomposition}
    f=\sum_\theta f_\theta=\sum_{\theta}(\psi_\theta f)=\sum_{\theta,v}(\eta_v)^{\vee}\ast(\psi_\theta f),
\end{equation}
where $\{\theta\}$ is a collection of finitely overlapping $R^{-1/2}$-balls in $B^2(0,1)$ and $\psi_\theta$ is a bump function associated to $\theta$; $v\in R^\frac{1+\de}{2}\ZZ^2$ and $\{\eta_v\}$ is a smooth partition of unity of $\ZR^2$ with compact Fourier support. Such wave packet decomposition can be found in \cite{Wang-restriction-R3} Section 2 (see also \cite{Guth-R3} Section 2).

Here is a key feature of wave packets: Given a function $f$ in $\ZR^n$, if $\wh{f}$ is supported in on a unit ball, then $|f|\Id_B$ is essentially constant (in the sense of averaging) for any unit ball $B\subset\ZR^n$. That is, one would expect $\|f\Id_B\|_p\approx \|f\Id_B\|_\infty$ for any $1\leq p\leq \infty$ (see for instance, \cite{Wang-restriction-R3} Lemma 2.6, 2.7). This leads to the following observation: $|Ef_{\theta}\Id_T|$ is essentially constant on every scale-$R$ tube $T$ in the form
\begin{equation}
    T=\{x=(x',x_3)\in B_R:|x'+2x_3c_\theta+v|\leq R^{1/2+\de}\},
\end{equation}
where $c_\theta$ is the center of $\theta$, and $v$ is any point in $\ZR^2$ (for example, $v\in R^{\frac{1+\de}{2}}\ZZ^2$).

\smallskip

What follows is a pruning of scale-$R$ wave packets of the function $f$. Let us introduce $O_\e(1)$ scales:
\begin{equation}
\label{scales-1}
    \rho_1=R^{1/2+\de}, \rho_2=R^{1/2+2\de}, \ldots,\rho_\ell=R
\end{equation}
with $\ell\sim 1/(2\de)=O_\e(1)$. The pruning of $f$ will be proceeded from the smallest scale $\rho_1$ to the biggest scale $\rho_\ell$. First let us prune wave packets at the smallest scale $\rho_1=R^{1/2+\de}$. By dyadic pigeonholing, we can find a dyadic number $\la_1$ so that
\begin{enumerate}
    \item There is a corresponding scale-$R$ directional set $\Theta_{\la_1}[R]$, and for each $\theta\in\Theta_{\la_1}[R]$, a scale-$R$ tube set $\ZT_{\la_1,\theta}[R]$. Denote also $\ZT_{\la_1}[R]=\bigcup_{\theta}\ZT_{\la_1,\theta}[R]$.
    \item Uniformly for each $\theta\in\Theta_{\la_1}[R]$,
    \begin{equation}
    \label{la-1}
        |\ZT_{\la_1,\theta}[R]|\sim \la_1.
    \end{equation}
    \item For each $T\in\ZT_{\la_1}[R]$, $\|f_T\|_2$ are about the same up to a constant multiple.
    \item The $L^p$ norm of the  sum of wave packets $\sum_{T\in\ZT_{\la_1}}Ef_T$ dominates the $L^p$ norm of $Ef$, in the sense that $\|\sum_{T\in\ZT_{\la_1}[R]}Ef_T\|_{L^p(B_R)}\gtrapprox\|Ef\|_{L^p(B_R)}$.
\end{enumerate}
Define $f_{\rho_1}=\sum_{T\in\ZT_{\la_1}[R]}f_T$ be the pruned function at the smallest scale $\rho_1$. This is the pruning at the first step.

Next, suppose there is a pruned function $f_{\rho_{j-1}}$ at scale $\rho_{j-1}$, we would like to prune it at scale $\rho_j$. By dyadic pigeonholing, there are dyadic numbers $\ka_1=\ka_1(j), \ka_2=\ka_2(j)$, and a collection of $\rho_j\times\rho_j\times R$ fat tubes $\wt\cT[\rho_j]$ so that
\begin{enumerate}
    \item Any two $\wt T_1,\wt T_2\in\wt\cT[\rho_j]$ either are parallel, or make an angle $\gtrsim \rho_j/R$.
    \item Each $\wt T\in\wt \cT[\rho_j]$ contains $\sim\ka_1$ 
    many $\rho_{j-1}\times\rho_{j-1}\times R$ fat tubes in $\wt\cT[\rho_{j-1}]$.
    \item For each directional cap $\theta'$ with  $d(\theta')=\rho_j/R^{1+\de}$, there are either $\sim \ka_2$ parallel $\rho_j\times\rho_j\times R$ fat tubes pointing to this direction, or no tubes at all.
    \item The above two items imply that for any $\wt T\in \wt\cT[\rho_j]$ with direction $\theta'$,
    \begin{equation}
    \label{l2-estimate-every-scale}
        \sum_{T\subset \wt T}\|f_T\|_2^2\lesssim\ka_2^{-1}\|f_{\theta'}\|_2^2.
    \end{equation}
    \item $\|\sum_{T\subset \bigcup_{\wt T\in\wt\cT[\rho_j]}}Ef_T\|_{L^p(B_R)}\gtrapprox\|Ef\|_{L^p(B_R)}$.
\end{enumerate}
Define $f_{\rho_j}=\sum_{T\subset \bigcup_{\wt T\in\wt\cT[\rho_j]}}f_T$ be the pruned function at scale $\rho_j$.  We remark that although the pruning at scale $\rho_j$ may destroy some uniform structure at scale $\rho_{j'}$ with $j'<j$ (for example, the fact that at scale $\rho_{j'}$ there are $\sim \ka_2(j')$  parallel $\rho_{j'}\times\rho_{j'}\times R$ fat tubes may no longer be true. Instead, there will be $\lesssim \ka_2(j')$ parallel fat tubes), the upper bound estimate \eqref{l2-estimate-every-scale} still remains true for all $j'<j$, namely, $\sum_{T\subset \wt T}\|f_T\|_2^2\lesssim\ka_2(j')^{-1}\|f_{\theta'}\|_2^2$ for every $\theta'$ with $d(\theta')=\rho_{j'}/R^{1+\de}$ and every fat tube $\wt T\in\wt\cT[\rho_{j'}]$ that points to $\theta'$.

The pruned function $f_{\rho_l}$ at the biggest scale $\rho_\ell$ is the one we will carefully study in the rest of the paper. For simplicity, we still denote $f=f_{\rho_l}$ by an abuse of notation. Let us emphasize some properties of the new $f$:
\begin{enumerate}
    \item For each scale $\rho_j$ in \eqref{scales-1} the following is true: For each $\rho_j/R^{1+\de}$ directional cap $\theta'$, there are $\lesssim \ka_2$   parallel $\rho_j\times\rho_j\times R$ fat tubes $\wt T$, each of which contains the same amount (up to a constant multiple) of thin tubes from $\ZT_{\la_1}[R]$, where $\ZT_{\la_1}[R]$ was defined near \eqref{la-1}.
    \item We have for each $\rho_j/R^{1+\de}$ directional cap $\theta'$ and $\ka_2=\ka_2(j)$,
    \begin{equation}
    \label{l2-estimate-pruning}
        \sum_{T\subset \wt T}\|f_T\|_2^2\lesssim\ka_2^{-1}\|f_{\theta'}\|_2^2.
    \end{equation}
\end{enumerate}

\subsection{Broad-narrow reduction}

The broad-narrow reduction was introduced in \cite{Bourgain-Guth-Oscillatory}. Specifically, partition the unit ball $B^2(0,1)$ into cubes $\{\tau\}$ such that $d(\tau)=K^{-1}$, where $K$ is a large number but is also small compared to $R$, for example, $K=(\log R)^{100\e^{-100}}$. Here is the formal definition of broadness (2-broad).

\begin{definition}[\cite{Guth-R3}]
\label{broad}
Given $A\geq1$ and any $x\in \ZR^3$, we define $\BR_A Ef(x)$ as the $[A]+1$ largest number in $\{|Ef_\tau(x)|\}_{\tau}$. That is, 
\begin{equation}
\label{alpha-broad}
    \BR_A Ef(x):=\min_{\tau_1,\ldots,\tau_{A}}\max_{\substack{\tau\not=\tau_k,\\1\leq k\leq [A]}} |Ef_\tau(x)|.
\end{equation}
\end{definition}
Eventually we will choose $A\leq K^{\e}$. Roughly speaking, the broad-narrow reduction allows us to focus on those points $x\in\ZR^3$ where considerably many $\{Ef_\tau(x)\}_{\tau}$ make major contribution to $Ef(x)$. A similar definition is given in \cite{Wang-restriction-R3} Section 2.2. After the broad-narrow reduction (see for instance \cite{Guth-R3}), we only need to consider the $\BL^p$-norm $\|Ef\|_{\BL^p(B_R)}^p:=\|\BR_A Ef\|_{L^p(B_R)}^p$.

\begin{remark}
\rm
The $\BL^p$-norm given here is slightly different from the one in \cite{Wang-restriction-R3}, but for our purpose they are the same. Some properties of the $\BL^p$-norm can be found in \cite{Wang-restriction-R3} Section 2.

\end{remark}

\subsection{Polynomial partitioning}

Let us recall the polynomial partitioning introduced in \cite{Guth-R3}. One can use Corollary 1.7 in \cite{Guth-R3} with a given degree $d$ to partition the measure $\mu_{Ef}(B_R)=\|Ef\|_{\BL^p(B_R)}^p$. The outcomes are
\begin{enumerate}
    \item A polynomial $P$ of degree $\sim d$.
    \item A collection of disjoint cells $\cu=\{U\}$ with $|\cu|\sim d^3$ such that 
    \begin{equation}
        B_R\setminus Z(Q)=\bigcup_{U\in\cu} U.
    \end{equation}
    \item $\mu_{Ef}(U)=\|Ef\|_{\BL^p(U)}^p$ are about the same up to a constant multiple for all $U\in\cu$.
    \item We can refined the polynomial partitioning a little bit to get an extra information on each $U\in \cu$: $U$ is contained in an $Rd^{-1}$-ball in $B_R$. This refinement was obtained in \cite{Wang-restriction-R3}.
\end{enumerate}
Now we introduce a wall $W$, which is a thin neighborhood of the variety $Z(Q)$:
\begin{equation}
    W:=N_{R^{1/2+\de}}(Z(Q)).
\end{equation}
Define for each $U\in\cu$ a smaller cell $O=U\setminus W$ and let $\co=\{O\}$. One advantage for looking at the smaller cell $O$ is that any tube of dimensions $R^{1/2+\de}\times R^{1/2+\de}\times R$ can only intersect at most $O(d)$ cells in $\co$. 

The decomposition above leads to a partition
\begin{equation}
    B_R=W\bigsqcup\Big(\bigsqcup_{O\in\co} O\Big)
\end{equation}
and hence an estimate
\begin{equation}
\label{partitioning-outcome-1}
    \|Ef\|_{\BL^p(B_R)}^p\lesssim\sum_{O\in\co}\|Ef\|_{\BL^p(O)}^p+\|Ef\|_{\BL^p(W)}^p.
\end{equation}
If the first term in \eqref{partitioning-outcome-1} dominates, then we say ``we are in the cellular case". Let $f_{O}$ be the sum of wave packets that intersect $O$. A crucial fact is 
\begin{equation}
\label{cell-pointwise}
    Ef_{O}(x)= Ef(x), \hspace{.5cm}x\in O.
\end{equation}
This yields, supposing that we are in the cellular case,
\begin{equation}
\label{cell-equation}
    \|Ef\|_{\BL^p(B_R)}^p\lesssim\sum_{O\in\co}\|Ef_{O}\|_{\BL^p(O)}^p.
\end{equation}

If the second term in \eqref{partitioning-outcome-1} dominates, then we say ``we are in the algebraic case". In this case we will divide wave packets into a transversal part and a tangential part. Introduce a collection of $R^{1-\de}$-balls $B$ in $B_R$. For each $B$, we define $S$ as the portion of the wall inside $B$: 
\begin{equation}
    S=W\cap B,
\end{equation}
and call it an ``algebraic fat surface". Then we define a tangential function $f_{S}$ and a transverse function $f_{S,trans}$ for each $B$. Roughly speaking, the tangential function $f_{S}$ contains all the wave packets $f_T$ that the $R$-tube $T$ is $R^{-1/2+2\de}$-tangent to $Z(Q)$ at each point of $Z(Q)\cap B$, and the transverse function $f_{S,trans}$ contains all the remaining wave packet. Here is the formal definition of tangential tubes:
\begin{definition}[\cite{Guth-R3}, Definition 3.3]
\label{tangential-def}
The collection of tangential tubes $\ZT_S$ is the set of tubes $T$ obeying the following two conditions:
\begin{enumerate}
    \item [$\bullet$] $T\cap S\not=\varnothing$.
    \item [$\bullet$] If $z$ is any non-singular point of $Z(P)$ lying in $2 B\cap 10T$, then 
    \begin{equation}
        |{\rm Angle}(v(T),T_zZ(P))|\leq R^{-1/2+2\de}.
    \end{equation}
\end{enumerate}
\end{definition}
\noindent Hence inside $S$ we have
\begin{equation}
    f=f_S+f_{S,trans},
\end{equation}
which gives that for $x\in S$,
\begin{equation}
\label{algebraic-eq}
    Ef(x)= Ef_{S,trans}(x)+ Ef_{S}(x).
\end{equation}

\smallskip

One can iteratively use the polynomial partitioning above for either cellular case and algebraic case to obtain a ``tree" structure.\footnote{We do not break the algebraic case into transverse case and tangent case when performing the iteration. This is different from the iteration (induction) given in \cite{Guth-R3}.} Each node of this tree is either a cellular cell or an algebraic cell. Note that in either cellular or algebraic case, the diameter of each resulting cell is strictly decreasing (decrease by either $1/d$ or $R^{-\de}$). The iteration will stop when the diameter of each cell is smaller than $R^\de$. A cell that appears the last step of the iteration is called a ``leaf". Here is a summary of this iteration:

\begin{lemma}[\cite{Wang-restriction-R3} Lemma 3.3]
\label{algebraic-lem-1}
For a function $f$ supported in $B^2(0,1)$, there is a tree structure $\co_{\text{tree}}$ of height $J\lesssim \log R/\log\log R$ satisfying the followings:
\begin{enumerate}
    \item The root of the tree $\co_{\text{tree}}$ is $O_0=B_R$.
    \item For each $0\leq j\leq J-1$, the children of a node $O_j$ of depth $j$ are some subsets $O_{j+1}$ of $O_j$, and each $O_{j+1}$ lies in a ball $B_{O_{j+1}}$ of radius $R_{j+1}\leq R_j/d$. The radius $R_{j+1}$ is called the scale of $O_{j+1}$, and the radii $R_{j+1}$ are the same for all $O_{j+1}$ of depth $j+1$. Moreover, there is a collection of $R_{j+1}$-tubes $\ZT_{O_{j+1}}$ associated to each $O_{j+1}$. 
    \item There is a number $n\leq\de^{-2}$ and indices $0=j_0<j_1<\cdots< j_n\leq J$ such that each node $O_{j_t}$ of depth $j_t,1\leq t\leq n$ is a portion of a fat $R_{j_t-1}$-surface, and $R_{j_t}=R_{j_t-1}^{1-\de}$. Here by a fat $r$-surface we mean a thin neighborhood $N_{r^{1/2+\de}}Z_S\cap B_S$, where $Z_S$ is a union of smooth algebraic surfaces with $\deg Z_S\leq d$, and $B_S$ is an $r^{1-\de}$-ball. To emphasis that $O_{j_t}$ is a portion of a fat surface, we denote $S_t=O_{j_t}$ and let $\cs_t$ be the collection of all $S_t$.
    \item The sets $\{N_{R_{j_t-1}^{1/2+\de}}(S_t)\}_{S_t\in\cs_t}$ are finitely overlapping.
    \item Let $T(j)=\max\{t:j_t\leq j\}$. For each $1\leq j\leq J$ we have
    \begin{equation}
    \label{partitioning-blp-norm2}
        \|Ef\|_{\BL^p(O_j)}^p\lesssim d^{-3(j-T(j))}\|Ef\|_{\BL^p(B_R)}^p
    \end{equation}
    as well as
    \begin{equation}
    \label{partitioning-blp-norm}
        \|Ef\|_{\BL^p(B_R)}^p\lesssim 2^j(\log R)^{T(j)}\sum_{O_j\in\co_{\text{tree}}}\|Ef\|_{\BL^p(O_j)}^p,
    \end{equation}
    where the sum $\sum_{O_j\in\co_{\text{tree}}}$ is over all nodes in the tree $\co_{\text{tree}}$ of depth $j$.
    \item If $j\not=j_t$, then for each node $O_{j-1}$ of depth $j-1$, a $R_{j-1}$-tube $T\in\ZT_{O_{j-1}}$ intersects $\lesssim d$  children $O_j$ of $O_{j-1}$.
    \item Denote by $\co_{leaf}$ the collection of leaves in the tree $\co_{tree}$. Then all sets in $\co_{leaf}$ are disjoint.
\end{enumerate}
\end{lemma}

The algebraic cases in the iteration are very important to us. Note that for a leaf $O'\in\co_{leaf}$, there are $n$ unique algebraic fat surfaces $S_1,S_2,\ldots, S_n$, so that $O'$ is within a containing chain
\begin{equation}
\label{containing-chain}
    O'\subset S_n\subset S_{n-1}\subset\cdots\subset S_1\subset B_R.
\end{equation}

Besides the outcomes from Lemma \ref{algebraic-lem-1}, we are in particular interested in some pointwise estimates similar to \eqref{cell-pointwise} and \eqref{algebraic-eq}. Note that for any $1\leq j\leq n$ and any ancestor $O_j$ of $S_j$ that $S_j\subset O_j\subset S_{j-1}$, by \eqref{cell-pointwise} one has
\begin{equation}
\label{pointwise-cell-2}
    Ef_{O_j}(x)= Ef_{S_{j-1},trans}(x),\hspace{.5cm}x\in O_j.
\end{equation}
Here we use the convention $S_0=B_R$ and $f_{S_0}=f$. 

We know that $f_{O_j}$ is essentially a sum of scale-$r_j$ ($r_j$ is the scale of $O_j$) wave packets, which are determined by tubes from, say $\ZT_{O_j}$. The dimensions of the tubes in $\ZT_{O_j}$ are roughly $r_j^{1/2}\times r_j^{1/2}\times r_j$. For any subcollection $\ZT\subset\ZT_{O_j}$, define $f^\ZT$
as the function concentrated on wave packets from $\ZT$. Recall \eqref{algebraic-eq}. The pointwise equality \eqref{pointwise-cell-2} further yields that
\begin{equation}
\label{iterative-equation}
    Ef_{O_j}^\ZT(x)= Ef^\ZT(x)-\Big(\sum_{l=1}^{j-1}Ef_{S_l}^\ZT(x)\Big), \hspace{.5cm}x\in O_j.
\end{equation}
If all $Ef_{S_l}^\ZT(x)$ are small, then similar to \eqref{cell-pointwise} we also have $Ef_{O_j}^\ZT(x)\sim Ef^\ZT(x)$. To find out when $Ef_{S_l}^\ZT(x)$ are small, we want to compare $Ef_{S_l}^\ZT$ with $Ef_{S_l}$. It is hard to get something meaningful for general $\ZT$, while if $\ZT$ is either a collection of transverse tubes or a collection of tangent tubes with respect to  an algebraic fat surface $\wt S$, it was showed in \cite{Wang-restriction-R3} Lemma 10.2 that  
\begin{equation}
\label{fat-fat}
    |Ef_{S_l}^\ZT(x)|\lesssim |Ef_{S_l}(x)|
\end{equation}
for any $x\in O_j\cap \wt S$.

This crucial fact gives that for any fat surface $S_t$ and $B_K\subset S_t$ (see Definition \ref{broad} for $K$), if $\|Ef_{S_l}\|_{\BL^p(B_K)}\lessapprox R^{-\Om(\de)}\|Ef_{S_t}\|_{\BL^p(B_K)}$ for all $1\leq l<t$, then
\begin{equation}
\label{fat-fat-cor}
    \|Ef_{S_t}\|_{\BL^p(B_K)}\sim\|Ef^{\ZT_{S_t}}\|_{\BL^p(B_K)},
\end{equation}
where $\ZT_{S_t}$ is the collection of tangent tube corresponding to the algebraic fat surface $S_t$ (note that $Ef_{S_t}=Ef_{S_t}^{\ZT_{S_t}}$ and see also Lemma 3.10 in \cite{Wang-restriction-R3}). We remark that \eqref{fat-fat} and \eqref{fat-fat-cor} also hold for $f$ replaced by $g$, where $g$ is a function similar to $f$ that will be given later in \eqref{g}.

\begin{remark}
\label{notation-remark}
\rm
The notations here are slightly different from the notations in \cite{Wang-restriction-R3} (in particular, Lemma 3.7 in \cite{Wang-restriction-R3}). We continue to use $f_S$ to denote the tangential function one obtained for the fat surface $S$ in polynomial partitioning iteration. While $f^{\ZT_S}$ means $f_{\Pi_{S}}$ in \cite{Wang-restriction-R3}, which essentially refers to restricting the original function $f$ to the wave packets from $\ZT_S$.

\end{remark}

We know by definition that tubes from $\ZT_S$ are contained in the fat surface $S$ with some scale $r$. While more is true because we are looking at the broad norm $\BL^p(S)$ on $S$. Roughly speaking, the following lemma allows us to assume that tubes in $\ZT_S$ are contained in the $r^{1/2+O(\de)}$-neighborhood of $r^{O(\de)}$ planes.
\begin{lemma}
\label{polynomial-plane-lemma}
There exist a collection of $\lesssim r^{O(\de)}$ planes $\cp$ so that if we let $\ZT_{S,\cp}\subset \ZT_S$ be the subcollection that each $r$-tube $T\in\ZT_{S,\cp}$ is contained in the $r^{1/2+O(\de)}$-neighborhood of some planes in $\cp$, then 
\begin{equation}
    \|Ef^{\ZT_S}\|_{\BL^p(S)}\lesssim \|Ef^{\ZT_{S,\cp}}\|_{\BL^p(S)}.
\end{equation}
\end{lemma}
\begin{proof}
Recall that $S\subset S_Z= N_{r^{1/2+\de}}(Z_S)\cap B_S$, where $Z_S$ is a finite union of transverse complete intersections, each of which has degree at most $d$, and $B_S$ is an $r^{1-\de}$-ball containing $S$. Partition $S$ into finitely overlapping $r^{1/2+\de}$-balls $\bq=\{q\}$. Denote by $\ZT_{S,q}$ the subcollection of tubes in $\ZT_S$ that intersects $q$:
\begin{equation}
    \ZT_{S,q}=\{T'\in\ZT_S,5T'\cap q\not=\varnothing\}.
\end{equation}
By Definition \ref{tangential-def}, there is a plane $Z_q$ so that tubes in $\ZT_{S,q}$ are all contained in a fat plane $P_q:=N_{r^{1/2+2\de}}(Z_q)$. We are interested in the intersection $S_Z\cap P_q$. Our (heuristic) goal is to prove the following dichotomy: either $\|Ef^{\ZT_S}\|_{\BL^p(q)}=0$, or $|S_Z\cap P_q|\gtrsim_K |N_{r^{1/2+\de}}(Z_q)\cap B_S|\sim r^{5/2-\de}$ (see Definition \ref{alpha-broad} for $K$). Note that the latter case cannot happen for too many almost distinct $P_q$. Otherwise it would violate $|S_Z|\lesssim d r^{5/2-\de}$, which is given by Wongkew's theorem.

\smallskip

Let $z\in Z_q\cap q$ be a point. After rotation and translation, let us assume that $z$ is the origin and $Z_q$ is the vertical plane $\Si=\{x_1=0\}$. The following set parameterizes a $r^{1/2+\de}$-tube whose center is $\ba$ and direction is $\bd$:
\begin{equation}
\nonumber
    T_{\ba,\bd}=\{(\bx,t)\in\ZR^2\times[0,cr^{1-\de}]:|\bx-\ba-t\bd|\leq r^{1/2+\de}\},\,\,(\ba,\bd)\in[0,r^{1/2+\de}]^2\times[0,4]^2.
\end{equation}
Here $c$ is a small absolute number (for example, $c$=1/100). By some appropriate rigid transform, we can assume that all tubes from $\ZT_S$ have a parameterization $T_{\ba,\bd}$.  Consider the set ($N_{r^{-1/2+3\de}}(\Si)\cap[0,4]^2$ is roughly a $1\times r^{-1/2+3\de}$-tube in $\ZR^2$)
\begin{equation}
\label{LS}
    L_{S}:=\{(\ba,\bd)\in[0,r^{1/2+\de}]^2\times (N_{r^{-1/2+3\de}}(\Si)\cap[0,4]^2):(0,0)\in T_{\ba,\bd}\subset S_Z\}.
\end{equation}
This set basically contains the union of tubes in $S_Z$ that intersect $(0,0)$ and is $R^{-1/2+2\de}$-tangent to the vertical plane $\Si$ ($L_S$ is morally $S_Z\cap P_q$). 

Denote $\Xi_1,\Xi_2$ the parameter spaces
\begin{equation}
    \Xi_1=[0,r^{1/2+\de}]^2\times (N_{r^{-1/2+3\de}}(\Si)\cap[0,4]^2),\hspace{.3cm}\Xi_2=\ZR^2\times[0,cr^{1-\de}].
\end{equation}
We claim that $L_S$ is semialgebraic with complexity at most $O_d(1)$. In fact,  if consider the following two sets
\begin{equation}
    Y_1=\{(\ba,\bd,\bx,t)\in\Xi_1\times\Xi_2:(\bx,t)\not\in S_Z,(\bx,t)\in T_{\ba,\bd}\}
\end{equation}
and
\begin{equation}
    Y_2=\{(\ba,\bd,\bx,t)\in\Xi_1\times\Xi_2:(0,0)\not\in T_{\ba,\bd}\},
\end{equation}
then $L_S$ is the complement of $\Pi_{(\ba,\bd)}(Y_1\cup Y_2)$ in $\Xi_1$, where $\Pi_{(\ba,\bd)}$ is the projection $(\ba,\bd,\bx,t)\mapsto(\ba,\bd)$. By Tarski's projection theorem (see \cite{Katz-Rogers}), $\Pi_{(\ba,\bd)}(Y_1\cup Y_2)$ has complexity $O_d(1)$, hence so is $L_S$.

Let $I=[0,4]$ be the interval that parameterizes directions in the vertical plane $\Si$. Consider the projection $P_2:(\ba,\bd)\mapsto\bd$ and $P_1:[0,4]^2\to I$, $P_1((x_1,x_2))=x_2$. Let $P=P_1\circ P_2$. By Tarski's projection theorem again, the projection $P(L_S)$ is a semialgebraic set of complexity $O_d(1)$. Hence $P(L_S)$ is a union of $O_d(1)$ disjoint intervals $\{I_j\}$ (a point is an interval of length zero). Notice that if $R$ is large enough, then the ``broad number" $A$ (see Definition \ref{broad}) is much larger than the number of intervals $\{I_j\}$.

Suppose $\|Ef^{\ZT_S}\|_{\BL^p(q)}\not=0$. Then by the definition of $\BL^p$-norm (Definition \ref{broad}) we know that there are at least $A$ tubes coming from $K^{-1}$-separate directions that intersect $q$, hence the origin. Since $A$ is larger the the number of intervals $\{I_j\}$, there is at least one interval $I_j$ with $|I_j|\geq K^{-1}$. The pull back $P^{-1}(I_j)\subset L_S\cap P_q$ is a union of $r^{1/2+\de}\times r^{1/2+\de}\times r^{1-\de}$-tubes rooted at $q$, each of which is $r^{-1/2+2\de}$-tangent to the vertical plane. Hence 
\begin{equation}
    r^{5/2-3\de}\lesssim|P^{-1}(I_j)|\leq|L_S\cap P_q|.
\end{equation}

Now we claim that there exists $r^{5\de}$ planes $\cp$, each of which is some plane $P_q$, so that $\bigcup_{q\in\bq}L_S\cap P_q$ is contained in the union $\bigcup_{P\in\cp}N_{r^{1/2+100\de}}(P)$. Note that this is enough to prove our lemma.

Suppose on the contrary that $\bigcup_{q\in\bq}L_S\cap P_q$ cannot be covered by less than $r^{5\de}$ fat planes $N_{r^{1/2+100\de}}(P)$ where $P\in\cp$. Let $|\cp|\sim r^{5\de}$. In particular, it means that for $P,P\in\cp$, $\ang(P,P')>r^{-1/2+90\de}$. Hence \begin{align}
    |L_S|\geq\Big|\bigcup_{P\in\cp} L_S\cap P\Big|&\geq\sum_{P\in\cp}|L_S\cap P|-\sum_{P,P'}|P\cap P'|\\ \nonumber
    &\gtrsim r^{5/2-3\de}|\cp|-r^{1/2-80\de}|\cp|^2\geq r^{5/2+\de}.
\end{align}
However, Wongkew's theorem gives $|L_S|\lesssim dr^{5/2-\de}$, yielding a contradiction.
\end{proof}

\subsection{Broom}
\label{subsection-broom}

As mentioned in the previous subsection, we obtain a tree after running the polynomial partitioning iteration for $\|Ef\|_{\BL^p(B_R)}^p$. Also we have 
\begin{equation}
    \|Ef\|_{\BL^p(B_R)}^p\lesssim R^{O(\de)}\sum_{O'\in\co_{leaf}}\|Ef\|_{\BL^p(O')}^p.
\end{equation}

In \cite{Wang-restriction-R3}, a relation ``$\sim$" between $R^{1-\e_0}$-balls ($\e_0=\e^{10}$) and $R$-tubes was introduced to make use of the broom structure. Roughly speaking, a scale-$R$ tube $T$ is related to an $R^{1-\e_0}$-ball $B_k$ if the wave packet $Ef_T$ is associated to a significant amount of fat algebraic surfaces inside $B_k$.  The formal definition of the relation is given in \cite{Wang-restriction-R3} Section 6. Using this relation, for fixed $B_k$, we can define a related function $f_k^\sim$, which is the sum all the wave packets $f_T$ that $T$ is related to $B_k$, and similarly for a unrelated function $f_k^{\not\sim}$. In other words, for a fixed ball $B_k$ if we define (recall \eqref{la-1})
\begin{equation}
\label{related-wp}
    \ZT_{k}^{\sim}[R]=\{T\in\ZT_{\la_1}[R]:T\text{ is related to }B_k\}
\end{equation}
as well as
\begin{equation}
\label{unrelated-wp}
    \ZT_{k}^{\not\sim}[R]=\{T\in\ZT_{\la_1}[R]:T\text{ is not related to }B_k\},
\end{equation}
then
\begin{equation}
\label{related-unrelated-fcn}
    f_k^\sim=\sum_{T\in\ZT_{k}^{\sim}[R]}f_T,\hspace{.5cm}f_k^{\not\sim}=\sum_{T\in\ZT_{k}^{\not\sim}[R]}f_T.
\end{equation}

There is a broom estimate \eqref{broom-esti-1} for the unrelated function $f_k^{\not\sim}$, which roughly says that the $R$-tubes rooted at a fat surface are only a small portion of all the $R$-tubes with similar directions. Indeed, suppose $S_t$ is fat surface at scale $R_{j_t}=r\leq R^{2/3}$ (see Lemma \ref{algebraic-lem-1} for $S_t$ and $R_{j_t}$) and $\ZT_{S_t,\theta'}[r]\subset\ZT_{S_t}[r]$ is the collection of $r$-tubes inside $\ZT_{S_t}[r]$ with direction $\theta'\in\Theta[r]$. Then 
\begin{equation}
\label{broom-esti-1}
    \|E(f_k^{\not\sim})^{\ZT_{S_t,\theta'}}\|_{L^2(B_{S_t})}^2\lesssim R^{O(\e_0)}\Big(\frac{r}{R}\Big)\cdot\Big(\frac{r}{R}\Big)^{1/2}\|Ef_{\theta'}\|_{L^2(\wt T_{S_t})}^2,
\end{equation}
where $\wt T_{S_t}$ is an $R/r^{1/2}\times R/r^{1/2}\times R$ fat tube containing $S_t$, pointing to the direction $\theta'$, and $B_{S_t}$ is the $r$-ball containing $S_t$ that we obtained from the polynomial partitioning iteration. One can compare \eqref{broom-esti-1} with the broom estimate (7.12) in \cite{Wang-restriction-R3}. Note that compared to the right hand side of (7.12) in \cite{Wang-restriction-R3}, the right hand side of \eqref{broom-esti-1} is integrated in a smaller region $\wt T_{S_t}$. This extra information is in fact given readily from the proof of Lemma 7.2 in \cite{Wang-restriction-R3}, since those $R$-tubes $T$ associated to $\ZT_{S_t,\theta'}$ (that is, $T\supset T'$ for some $T'\in\ZT_{S_t, \theta'}$,  and the directional cap of $T$ is contained in $\theta'$) are all contained in $\wt T_{S_t}$.

Recall the wave packet pruning in Section \ref{wp-pruning-section}. Suppose $r\leq R^{2/3}$ and suppose $R/r^{1/2}\leq\rho_j$ but $R/r^{1/2}\geq\rho_{j-1}$ for some scale $\rho_j$ in \eqref{scales-1}. Then as a consequence of \eqref{l2-estimate-pruning} and \eqref{broom-esti-1},
\begin{equation}
\label{broom-esti-2}
    \|E(f_k^{\not\sim})^{\ZT_{S_t,\theta'}}\|_{L^2(B_{S_t})}^2\lesssim R^{O(\e_0+\de)}\Big(\frac{r}{R}\Big)^{3/2}\ka_2(j)^{-1}\|Ef_{\theta'}\|_2^2.
\end{equation}
Comparing \eqref{broom-esti-2} to (7.12) in \cite{Wang-restriction-R3}, there is an extra gain $\ka_2(j)^{-1}$.

\begin{remark}
\label{broom-remark-0}

\rm

Let $\ZT_{k}'[R]\subset\ZT_k^{\not\sim}[R]$ and $f_k'=\sum_{T\in\ZT_{k}'[R]} f_T$. We remark that $f_k'$, the sum of a subcollection of unrelated wave packets, still satisfies the broom estimates \eqref{broom-esti-1} and \eqref{broom-esti-2} (with $f_k^{\not\sim}$ replaced by $f_k'$). Indeed, the broom estimate \eqref{broom-esti-1} follows from the ingredient: Suppose $(f_k^{\not\sim})^{\ZT_{S_t,\theta'}}$ is concentrated on $\cb$, a broom rooted at $S_t$ containing $b$ scale-$R$ wave packets (see Definition 5.4 in \cite{Wang-restriction-R3} for a broom). Then inside $\wt T_{S_T}$ there are at least $b^2$ scale-$R$ wave packets in $\ZT_{\la_1}[R]$ (recall \eqref{la-1}). Hence by Lemma 5.10 in \cite{Wang-restriction-R3} we have \eqref{broom-esti-1}.

Now since $\ZT_{k}'[R]\subset\ZT_k^{\not\sim}[R]$, $(f_k')^{\ZT_{S_t,\theta'}}$ is concentrated on a smaller broom $\cb'\subset\cb$. Thus by Lemma 5.10 in \cite{Wang-restriction-R3} again we have the broom estimate \eqref{broom-esti-1} with $f_k^{\not\sim}$ replaced by $f_k'$. Similar reasoning also applies for \eqref{broom-esti-2}.

\end{remark}

In Section \ref{two-ends-section}, we will introduce a new relation ``$\sim_{\bn}$" among the $R$-tubes $T$ and the $R^{1-\e_0}$-balls $B_k$ based on some Kakeya structures. Eventually for each $R^{1-\e_0}$-ball $B_k$, we will define an ``ultimate" unrelated function---the sum of wave packets $f_T$ that $T\not\sim B_k$ as well as $T\not\sim_{\bn} B_k$. This unrelated function enjoys both the broom estimate \eqref{broom-esti-2} (see Remark \ref{broom-remark-0}) and some additional Kakeya estimates.

\subsection{A ``planar" bilinear estimate}

The main result in this subsection (Lemma \ref{planar-refinement-lem}) is essentially two-dimensional. While for our later purpose, we instead state it under the three-dimensional setting.

Suppose that $\Ga\subset B^3(0,1)$ is a truncated planar $C^2$ curve with positive two-dimensional second fundamental form. Suppose also that $\Ga_1,\Ga_2\subset\Ga$ are two sub-curves with $\dist(\Ga_1,\Ga_2)\sim1$. Let $R\gg\rho\gg R^{1/2}\gg1$ be two large numbers. The classical bilinear theory says that if $g_1$, $g_2$ are two functions whose Fourier transforms are supported in $N_{R^{-1}}(\Ga_1)$, $N_{R^{-1}}(\Ga_2)$ respectively, then (see also \eqref{cordoba-ferfferman})
\begin{equation}
\label{planar-bilinear}
    \int_{S}|g_1g_2|^2\lesssim R^{-2}\rho^{-1}\|g_1\|_2^2\|g_2\|_2^2,
\end{equation}
where $S=N_{\rho}(P)\cap B_R$ with the plane $P$ being parallel to the planar curve $\Ga$.

\begin{lemma}
\label{planar-refinement-lem}
Let $g_1,g_2$ be defined at the beginning of this subsection, and let $\cq$ be a collection of disjoint $\rho$-balls contained in $S$ ($S$ was introduced below \eqref{planar-bilinear}) with $R^{1/2}\leq\rho\leq R$. 
Let $F: \mathcal{Q}\rightarrow \mathbb{R}^+$ be that $F(Q)$ are about the same up to a factor of $R^{\delta}$ for all $Q\in \cq$, and that $F(Q)\leq \int_Q|g_1g_2|^{p/2}$. In addition, 
the decomposition of $g_1,g_2$ in \eqref{decomposition-g-j} satisfy the uniform-incidence assumption \eqref{wpt-uniform-incidence}. Then
\begin{equation}
\label{planar-refinement}
    \sum_{Q\in \mathcal{Q}} F(Q) \lesssim R^{O(\delta)}\rho^{3-p}R^{-p/2}|\cq|^{1-p/2} \min\{|\cq|^{p/4},\eta^{p/2}\} \|g_1\|_2^{p/2}\|g_2\|_2^{p/2}.
\end{equation}
\end{lemma}

\begin{proof}

By H\"older's inequality and the bilinear estimate \eqref{planar-bilinear} one has
\begin{align}
\label{planar-refinement-1}
    \int_{\cup_{Q\in\cq}Q}|g_1g_2|^{p/2}&\lesssim(\rho^3|\cq|)^{1-p/4}\Big(\int_{\cup_{Q\in\cq}Q}|g_1g_2|^2\Big)^{p/4}\\  \nonumber
    &\lesssim  \rho^{3-p}R^{-p/2}|\cq|^{1-p/4}\|g_1\|_2^{p/2}\|g_2\|_2^{p/2}.
\end{align}
Estimate \eqref{planar-refinement-1} is indeed sharp, unless there are some additional assumptions on the set $\bigcup_{Q\in\cq}Q$ and the functions $g_1,g_2$.

Decompose $g_1,g_2$ as (such decomposition can be viewed as a scale-$\rho^{-1}\times\rho^{-1}\times R^{-1}$ wave packet decomposition of $g_1$ and $g_2$. See also Section \ref{wp-pruning-section})
\begin{equation}
\label{decomposition-g-j}
    g_1=\sum_{T\in\ZT_1}g_{1,T},\hspace{.5cm}g_2=\sum_{T\in\ZT_2}g_{2,T},
\end{equation}
where $\ZT_j$ is a collection planar $\rho\times\rho\times R$-tubes contained in $S=N_\rho(P)\cap B_R$, and $\wh{g_{j,T}}$ is supported in a $\rho^{-1}\times\rho^{-1}\times R^{-1}$-cap, whose shortest direction is the same as the longest direction of $T$.

Suppose every $T\in\ZT_1\cup\ZT_2$ satisfies the following uniform-incidence assumption
\begin{equation}
\label{wpt-uniform-incidence}
    |\{Q\in\cq:Q\cap T\}|\sim \eta.
\end{equation}
Then for each $Q\in\cq$, by H\"older's inequality and the C\'ordoba-Fefferman $L^4$ orthogonality,  
\begin{align}
\label{cordoba-ferfferman}
    \int_{Q}|g_1g_2|^{p/2}&\lesssim \rho^{3-3p/4}\Big(\int_Q|g_1g_2|^2\Big)^{p/4}\\ \nonumber
    &\lesssim \rho^{3-3p/4}\Big(\sum_{T_1\in\ZT_1}\sum_{T_2\in\ZT_2}\int_{2Q}|g_{1,T}g_{2,T}|^2\Big)^{p/4}\\ \nonumber
    &\lesssim \rho^{3-p}R^{-p/2}\Big(\sum_{\substack{T_1\in\ZT_1,\\T_1\cap 2Q\not=\varnothing}}\sum_{\substack{T_2\in\ZT_2,\\T_2\cap 2Q\not=\varnothing}}\|g_{1,T}\|_2^2\|g_{2,T}\|_2^2\Big)^{p/4}.
\end{align}
Note that from the uniform-incidence assumption \eqref{wpt-uniform-incidence},
\begin{align}
\label{l2-from-uniform-incidence}
    &\sum_{Q\in\cq}\Big(\sum_{\substack{T_1\in\ZT_1,\\T_1\cap 2Q\not=\varnothing}}\sum_{\substack{T_2\in\ZT_2,\\T_2\cap 2Q\not=\varnothing}}\|g_{1,T}\|_2^2\|g_{2,T}\|_2^2\Big)^{1/2}\\ \nonumber
    \leq &\,\Big(\sum_{Q\in\cq}\sum_{\substack{T_1\in\ZT_1,\\T_1\cap 2Q\not=\varnothing}}\|g_{1,T}\|_2^2\Big)^{1/2}\Big(\sum_{Q\in\cq}\sum_{\substack{T_2\in\ZT_2,\\T_2\cap 2Q\not=\varnothing}}\|g_{2,T}\|_2^2\Big)^{1/2}\lesssim \eta \|g_1\|_2\|g_2\|_2.
\end{align}

Now we will make use of the assumption that $F(Q)$ are about the same up to a factor $R^{\delta}$. Indeed, by pigeonholing in \eqref{l2-from-uniform-incidence}, there is a particular $Q$ so that
\begin{equation}
    \Big(\sum_{\substack{T_1\in\ZT_1,\\T_1\cap 2Q\not=\varnothing}}\sum_{\substack{T_2\in\ZT_2,\\T_2\cap 2Q\not=\varnothing}}\|g_{1,T}\|_2^2\|g_{2,T}\|_2^2\Big)^{1/2}\lesssim R^{O(\delta)} |\cq|^{-1}\eta \|g_1\|_2\|g_2\|_2.
\end{equation}
Plug it back to \eqref{cordoba-ferfferman} with this particular $Q$ so that
\begin{align}
\label{planar-refinement-2}
    \sum_{Q\in \mathcal{Q}} F(Q) &\lesssim R^{O(\delta)}|\cq|\int_{Q}|g_1g_2|^{p/2}\\ \nonumber
    &\lesssim R^{O(\delta)}\rho^{3-p}R^{-p/2}|\cq|\Big(\sum_{\substack{T_1\in\ZT_1,\\T_1\cap 2Q\not=\varnothing}}\sum_{\substack{T_2\in\ZT_2,\\T_2\cap 2Q\not=\varnothing}}\|g_{1,T}\|_2^2\|g_{2,T}\|_2^2\Big)^{p/4}\\ \nonumber
    &\lesssim R^{O(\delta)}\rho^{3-p}R^{-p/2}(|\cq|^{1-p/2}\eta^{p/2})\|g_1\|_2^{p/2}\|g_2\|_2^{p/2}.
\end{align}
We can conclude the proof by combining  \eqref{planar-refinement-1} and \eqref{planar-refinement-2}. 
\end{proof}


\subsection{A refined Wolff's hairbrush result}

In \cite{Wolff-Kakeya}, Wolff used a geometric structure called ``hairbrush" to obtain the 5/2 bound for the three-dimensional Kakeya maximal conjecture. Roughly speaking, suppose $\cT$ is a collection of $\de\times\de\times1$ tubes with $\de$-separated directions, and suppose for each tube $T\in\cT$ there is an associated shading $Y(T)\subset T$ with a uniform density assumption $|Y(T)|/|T|\sim\la$ where $\de\ll\la\ll1$. Then for any $\e>0$, 
\begin{equation}
\label{Wolff-result}
    \Big|\bigcup_{T\in\cT}Y(T)\Big|\geq c_\e\de^{\e+1/2}\la^{5/2}(\de^2|\cT|).
\end{equation}
While a careful analysis suggests that \eqref{Wolff-result} is sharp only if for each tube $T\in\cT$, the shading $Y(T)$ is concentrated on one end of $T$. Thus, the union $\bigcup_{T\in\cT}Y(T)$ can be a much larger set if the shading $Y(T)$ satisfies some two-ends condition (see Figure \ref{two-ends-figure}). 
\begin{figure}

\begin{tikzpicture}

\node[shift={(-2cm,0cm)}, rotate=50,x={(0cm,0cm)}, cylinder, 
    draw = blue, 
    text = purple,
    cylinder uses custom fill, 
    minimum width = .5cm,
    minimum height = 8cm] (c) at (0,0) {};
    
\node[shift={(-2cm,0cm)}, rotate=20,x={(0cm,0cm)}, cylinder, 
    draw = blue, 
    text = purple,
    cylinder uses custom fill, 
    minimum width = .5cm,
    minimum height = 8cm] (c) at (0,0) {};

\node[shift={(3cm,0cm)},rotate=30,x={(0cm,0cm)}, cylinder, 
    draw = blue, 
    text = purple,
    cylinder uses custom fill, 
    minimum width = .5cm,
    minimum height = 8cm] (c) at (0,0) {};
    
\node[shift={(3cm,0cm)},rotate=60,x={(0cm,0cm)}, cylinder, 
    draw = blue, 
    text = purple,
    cylinder uses custom fill, 
    minimum width = .5cm,
    minimum height = 8cm] (c) at (0,0) {};

\node[cylinder, 
    draw = blue, 
    text = purple,
    cylinder uses custom fill, 
    minimum width = .5cm,
    minimum height = 8cm] (c) at (0,0) {};

\fill [gray] (-2cm,0cm) circle (.24cm);
\fill [gray] (-1.8cm,0cm) circle (.24cm);
\fill [gray] (-2.2cm,0cm) circle (.24cm);

\fill [gray] (0cm,0cm) circle (.24cm);
\fill [gray] (.2cm,0cm) circle (.24cm);
\fill [gray] (-.3cm,0cm) circle (.24cm);

\fill [gray] (3cm,0cm) circle (.24cm);
\fill [gray] (3.2cm,0cm) circle (.24cm);
\fill [gray] (3.3cm,0cm) circle (.24cm);

\fill [gray] (-.3cm,2cm) circle (.24cm);
\fill [gray] (-.4cm,1.9cm) circle (.24cm);
\fill [gray] (-.2cm,2.1cm) circle (.24cm);

\fill [gray] (1.2cm,1.2cm) circle (.24cm);
\fill [gray] (1.1cm,1.1cm) circle (.24cm);
\fill [gray] (1cm,1.1cm) circle (.24cm);

\fill [gray] (4.7cm,3cm) circle (.24cm);
\fill [gray] (4.7cm,2.9cm) circle (.24cm);
\fill [gray] (4.6cm,2.7cm) circle (.24cm);

\fill [gray] (5.6cm,1.5cm) circle (.24cm);
\fill [gray] (5.8cm,1.6cm) circle (.24cm);
\fill [gray] (5.4cm,1.4cm) circle (.24cm);

\end{tikzpicture}

\caption{quantitative two-ends}
\label{two-ends-figure}

\end{figure}
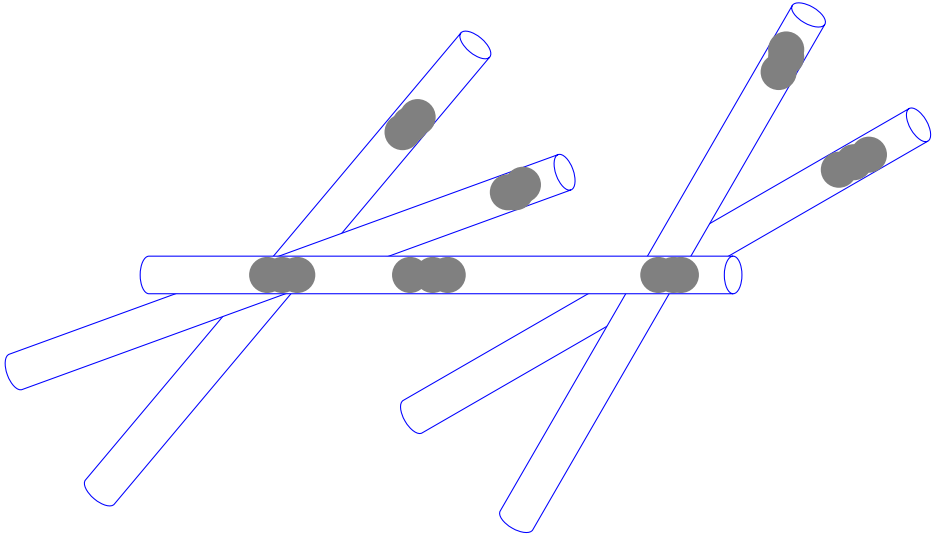

\begin{definition}[quantitative two-ends at scale $\epsilon'$]
\label{two-ends-def}
Let $\cT$ be a collection of $\de\times\de\times1$ tubes in $\ZR^3$. Each $T\in\cT$ has shading $Y(T)\subset T$. Suppose we can partition $T$ into $\de^{-\epsilon}$ many $\de^{\epsilon}$-segments 
$T_j$. We say $Y(T)$ is quantitative two-ends at scale $\epsilon'$ ($\epsilon>\epsilon'>\epsilon^{100}>0$) if
\begin{enumerate}
    \item For those $Y(T_j)\not=\varnothing$, $|Y(T_j)|$ are about the same up to a constant multiple.
    \item For each $T$, the number nonempty segments $Y(T_j)$ is bounded below by $\de^{-\epsilon'}$.
\end{enumerate}
\end{definition}

\begin{remark}
\label{two-end-remark}
\rm

If $Y(T)$ is quantitative two-ends at scale $\epsilon'$ and $Y'(T)\subset Y(T)$ is a subset satisfying $|Y'(T)|\gtrapprox|Y(T)|$, then there is a subset of $Y'(T)$ that is also quantitative two-ends at a smaller scale $\epsilon'/2$. This is because by pigeonholing we can find $\gtrapprox\de^{-\epsilon'}$ segments $T_j$ so that $|Y'(T_j)|$ are about the same up to a constant multiple, and their sum is $\gtrapprox|Y(T)|$. We will use this observation in the proof of next Lemma.
\end{remark}

\begin{lemma}
\label{two-ends-hairbrush}
Suppose that $\cT$ is a collection of $\de\times\de\times1$ tube in $\ZR^3$ pointing in $\de$-separated directions, and there are $\lesssim m$ parallel tubes in each direction. Suppose also there is a $\la<1$ so that for each $T\in\cT$, there is a shading $Y(T)\subset T$ satisfying $|Y(T)|\sim\la|T|$. Moreover, the shading $Y(T)$ is quantitative two-ends at scale $\epsilon'$ for some $\epsilon'>0$. Define $A=\bigcup_{T\in\cT} Y(T)$. Then $|A|\gtrapprox \de^{1/2}m^{-1}\la\sum_{T}|Y(T)|\sim\de^{1/2}m^{-1}\la^2(\de^2|\cT|)$.

\smallskip

As a direct corollary, note that by dyadic pigeonholing, there exists a dyadic number $\mu\geq1$ and a subset $A_\mu\subset A$ so that
\begin{enumerate}
    \item Every point in $A_\mu$ intersects $\sim\mu$ many $Y(T)$.
    \item $\sum_{T}|Y(T)|\geq\mu|A_\mu|\gtrapprox \sum_{T}|Y(T)|\sim\la(\de^2|\cT|)$.
\end{enumerate}
Also note that one can apply the bound $|A|\gtrapprox \de^{1/2}m^{-1}\la^2(\de^2|\cT|)$ to obtain $|A_\mu|\gtrapprox\de^{1/2}m^{-1}\la^2(\de^2|\cT|)$ by {\rm(2)} and some dyadic pigeonholing. Hence we also have
\begin{enumerate}
    \item[(i)] $\mu\la\lessapprox \de^{-1/2}m$.
    \item[(ii)] If we let $A_{\mu'}\subset A$ be the subset so that every point in $A_{\mu'}$ intersects $\sim\mu'$ many $Y(T)$, then $|A_{\mu'}|\lessapprox \mu|A_{\mu}|/\mu'$.
\end{enumerate}
\end{lemma}

\begin{proof}

Recall that $A=\bigcup_{T\in\cT}Y(T)$. What follows is a  ``quantitative broad-narrow" reduction on every point $p\in A$.  Define for each $p\in A$ a set of tubes 
\begin{equation}
    \cT(p)=\{T\in\cT:p\in Y(T)\},
\end{equation}
and for any directional cap $\si\subset\ZS^2$ a subset
\begin{equation}  
    \cT_\si(p)=\{T\in\cT(p):\text{the direction of }T\text{ belongs to }\si\}.
\end{equation}

\begin{lemma}
\label{two-ends-lem}
Let $K_0=e^{(\log\de^{-1})^{1/2}}$ and $K_1=(\log\rho^{-1})^{\e^{-10}}$. Suppose $|\cT(p)|\geq \de^{-\e/2}$. Then there are a scale $\rho$ with $\de\leq \rho\leq 1$, a directional cap $\si\subset\ZS^2$ with $d(\si)\sim  \rho$, a collection of directional caps $\Om$ with $|\Om|\geq K_1$, so that $|\cT_\om(p)|$ are about the same up to a constant multiple for $\om\in\Om$. Moreover,  $\om\subset\si$, $d(\om)\sim \rho/K_0$ for any $\om\in\Om$,  $\sum_{\om\in\Om}|\cT_\om(p)|\gtrapprox|\cT_\si(p)|$, and 
\begin{equation}
\label{two-ends-eq}
    |\cT_\si(p)|\gtrapprox|\cT(p)|.
\end{equation}
\end{lemma}

\begin{proof}
The proof is quite standard. We will prove the lemma by an iterative argument. First partition $\ZS^2$ into $\sim K_0^2$ many $1/K_0$-caps $\om_1$. By pigeonholing there is a subcollection $\Om_1$ so that $|\cT_{\om_1}(p)|$ are about the same up to a constant multiple for $\om_1\in\Om_1$, and
\begin{equation}
    \sum_{\om_1\in\Om_1}|\cT_{\om_1}(p)|\gtrsim (\log \de^{-1})^{-1}|\cT(p)|.
\end{equation}
If $|\Om_1|\geq K_1$, then we stop. Otherwise, pick any $\om_1\in\Om_1$, partition it into $\sim K_0^2$ many $1/K_0^2$-caps $\om_2$ and repeat the argument above. At the $m$-th step of the iteration (if the iteration does not stop at step $m$), we have a collection of caps $\Om_m$ and a $1/K_0^{m-1}$ cap $\om_{m-1}\in\Om_{m-1}$ with $d(\om_{m-1})\sim 1/K_0^{m-1}$ so that for every $\om_m\in\Om_m$, $d(\om_m)\sim 1/K_0^m$ and $\om_m\subset\om_{m-1}$, as well as
\begin{equation}
    \sum_{\om_m\in\Om_m}|\cT_{\om_m}(p)|\gtrsim (\log \de^{-1})^{-m}|\cT(p)|.
\end{equation}
Since $|\cT(p)|\geq\de^{-\e/2}$, the iteration will eventually stop at some step $n\lesssim (\log \de^{-1})^{1/2}$ with $1/K_0^n\geq\de$. Denote by $\si=\om_{n-1}$, $\rho=1/K_0^{n-1}$, and $\Om=\Om_n$ so that $|\Om|\geq K_1$. To see \eqref{two-ends-eq} holds, note that $|\cT(p)|\lesssim (\log \de^{-1})^{(\log\de^{-1})^{1/2}}|\cT_\si(p)|\lessapprox|\cT_\si(p)|$.
\end{proof}

For any point $p\in A$ with $|\cT(p)|\geq\rho^{-\e/2}$, by Lemma \ref{two-ends-lem} define
\begin{equation}
    \ang(p)=r \text{   and   }\si(p)=\si.
\end{equation}
If we already have $|A|\gtrapprox\de^{1/4}m^{-1}\la^2(\de^2|\cT|)$ then there is nothing to prove. Otherwise, for those $p\in A$ with $|\cT(p)|\geq \de^{-\e/2}$, define 
\begin{equation}
    A_\rho=\{p\in A: \ang(p)=\rho\}.
\end{equation}
Note that the number of possible choices of dyadic $\rho$ is $\lesssim\log\de^{-1}$. By Lemma \ref{two-ends-lem} and a pigeonholing on the dyadic numbers $\de\leq \rho\leq 1$, there exists a $\rho$ so that 
\begin{equation}
    \sum_{p\in A_\rho}|\cT_{\si(p)}(p)|\gtrapprox\sum_{p\in A_\rho}|\cT(p)|\gtrapprox \sum_{T\in\cT}|Y(T)|.
\end{equation}
Define a new shading
\begin{equation}
    Y_{\rho}(T)=\{p\in Y(T)\cap A_\rho: T\in\cT_{\si(p)}(p)\}.
\end{equation}

\begin{remark}
\label{two-ends-remark}

\rm

Every point $p\in A_\rho$ is now quantitative $\rho/K_0$-broad with respect to the shading $Y_\rho$. That is, the set $\cT_{\rho}(p)=\{T\in\cT:p\in Y_{\rho}(T)\}$ can be partition into $\geq K_1$ subcollections $\{\cT_{\rho,\om}(p)\}_\om$ so that $|\cT_{\rho,\om}(p)|$ are about the same, and for most pairs $(\om,\om')$, tubes in $\cT_{\rho,\om}$ and $\cT_{\rho,\om'}$ are $\rho/K_0$-transverse. The quantitative lower bound $\geq K_1$ is crucial, since it implies that if $\cT_{\rho}'(p)\subset\cT_{\rho}(p)$ is a subset that contains a fraction $\gtrsim (\log\rho^{-1})^{-100}$ tubes in $\cT_{\rho}(p)$, then most pairs of tubes in $\cT_{\rho}'(p)$ are still $\rho/K_0$-transverse. 

\end{remark}

Let $\Si$ be a collection of finitely overlapping $\rho$-caps of $\ZS^2$. For each $\si\in\Si$, let $\bar\cT_\si$ be a collection of parallel $\rho$-tubes pointing to the directional cap $\si$ that forms a finitely overlapping cover of the unit ball. Denote by $\bar\cT=\bigcup_{\si\in\Si}\bar\cT_\si$. For any $\bar T\in\bar\cT$, let
\begin{equation}
    \cT(\bar T)=\{T\in\cT:T\subset \bar T\}.
\end{equation}
Note that for each $p\in A_\rho$ there is a $\bar T\in\bar\cT$ depending on $p$ so that
\begin{equation}
\label{T-r-p}
    \cT_\rho(p)=\cT_{\si(p)}(p)=\{T\in\cT:p\in Y_{\rho}(T)\}=\{T\in\cT(\bar T):p\in Y_{\rho}(T)\}.
\end{equation}
Hence by Lemma \ref{two-ends-lem} and pigeonholing again there is a subset $\cT_\rho\subset\cT$ with $|\cT_\rho|\gtrsim (\log\rho^{-1})^{-2}|\cT|$ so that
\begin{enumerate}
    \item $|Y_{\rho}(T)|\gtrapprox \la|T|$ for all $T\in\cT_\rho$. 
    \item $\sum_{T\in\cT_\rho}|Y_\rho(T)|\gtrsim (\log\rho^{-1})^{-2}\sum_{p\in A_\rho}|\cT_{\si(p)}(p)|=(\log\rho^{-1})^{-2}\sum_{p\in A_\rho}|\cT_{\rho}(p)|$.
    \item For each $\bar T\in\bar\cT$, either $\cT_\rho(\bar T)=\varnothing$ or $|\cT_\rho(\bar T)|\gtrsim  (\log\rho^{-1})^{-2}|\cT(\bar T)|$, where $\cT_\rho(\bar T)$ is defined as
    \begin{equation}
        \cT_\rho(\bar T)=\{T\in\cT_\rho:T\subset \bar T\}.
    \end{equation}
\end{enumerate}
Denote by
\begin{equation}
\label{r-broad-part}
    A_\rho(\bar T)=\bigcup_{T\in\cT_\rho(\bar T)}Y_{\rho}(T).
\end{equation}
By the definition of $A_\rho$ we know 
\begin{equation}
    \sum_{\bar T\in\bar\cT}|A_\rho(\bar T)|\lessapprox |A_\rho|.
\end{equation}
Hence to prove the lemma, it suffices to prove that for any $\bar T\in\bar\cT$ 
\begin{equation}
\label{pf-two-end-reduction}
    |A_\rho(\bar T)|\gtrapprox\de^{1/2}m^{-1}\la^2(\de^2|\cT_\rho(\bar T)|).
\end{equation}
Let us fix a $\bar T\in\bar\cT$ from now on.

For each $p\in A_\rho(\bar T)$, recall $\cT_\rho(p)=\{T\in\cT(\bar T):p\in Y_\rho(T)\}$ in \eqref{T-r-p}. Since $|\cT_\rho(\bar T)|\gtrsim  (\log\rho^{-1})^{-2}|\cT(\bar T)|$, by pigeonholing there are a dyadic number $\nu>1$ and a subset $A'(\bar T)\subset A_\rho(\bar T)$ with 
\begin{equation}
    \int_{A'(\bar T)}|\cT'(p)|dp\gtrsim  (\log\rho^{-1})^{-2} \sum_{T\in\cT_\rho(\bar T)}|Y_\rho(T)|\gtrapprox(\log\rho^{-1})^{-4}\sum_{p\in A_r}|\cT_{\rho}(p)|.
\end{equation}
so that for all $p\in A'(\bar T)$, $|\cT'(p)|\sim \nu\gtrsim  (\log\rho^{-1})^{-4} |\cT_\rho(p)|$. Note that by Remark \ref{two-ends-remark}, every point in $A'(\bar T)$ is still quantitative $\rho$-broad with respect to the shading $Y_\rho$. Define a new shading
\begin{equation}
    Y'(T)=Y_\rho(T)\cap A'(\bar T).
\end{equation}
By pigeonholing there is a subset $\cT'(\bar T)\subset\cT_\rho(\bar T)$ with $|\cT'(\bar T)|\gtrapprox|\cT_\rho(\bar T)|$ so that $|Y'(T)|\gtrapprox \la|T|$ for all $T\in\cT'(\bar T)$.

Thus, on one hand
\begin{equation}
\label{hairbrush-esti-1}
     \nu|A_\rho(\bar T)|\gtrsim\nu|A'(\bar T)|\gtrapprox \sum_{T\in \cT'(\bar T)}|Y(T)|\approx\la(\de^2|\cT_\rho(\bar T)|).
\end{equation}
On the other hand, pick any $T\in\cT'(\bar T)$. Since every point in $Y'(T)$ is quantitative $\rho/K_0$-broad with respect to the shading $Y_\rho$, and note that $Y_\rho(T)$ is quantitative two-ends (see Remark \ref{two-end-remark}) for every tube $T\in\cT_\rho(\bar T)$. By Wolff's hairbrush argument and the two-dimensional X-ray estimate, we get
\begin{equation}
    |A_\rho(\bar T)|\gtrapprox\la^3\nu\de m^{-1}\rho^{-1}.
\end{equation}
Since $\rho<1$ and since $\de^2|\cT_\rho(\bar T)|\lesssim m$ by assumption, the above two estimates imply 
\begin{equation}
    |A_\rho(\bar T)|\gtrapprox\de^{1/2}m^{-1}\la^2(\de^2|\cT_\rho(\bar T)|),
\end{equation}
which is what we need in \eqref{pf-two-end-reduction}.
\end{proof}

\section{Quantitative two-ends reductions}
\label{two-ends-section}
In this section we conduct two quantitative two-ends reductions to the tube set $\ZT_{\la_1}[R]$ (see \eqref{la-1}). The two reductions will be applied to two different methods of bounding $\sum_{O'\in\co_{leaf}}\|Ef\|_{\BL^p(O')}^p$, respectively. Before elaborating on the reductions, let us briefly demonstrate our two methods. Along the way we will recognize the need of the two-ends reductions---it strengthens some Kakeya estimates.

\subsection{Brief outline for the two methods}
\label{brief-outline}
The first method deals with the case $R^{1/2}\leq r\leq R^{2/3}$. Suppose that for each scale-$r$ fat surface $S_t\in\cs_t$ (see Lemma \ref{algebraic-lem-1} for $\cs_t$), there are $\sim \si_1$ scale-$r$ directional caps that contribute to $\ZT_{S_t}$. Hence by $L^2$ orthogonality (see \eqref{related-unrelated-fcn} for $f_k^{\not\sim}$)
\begin{equation}
\label{estimate-a}
    \|(f_k^{\not\sim})^{\ZT_{S_t}}\|_{2}^2\lesssim \si_1\sup_{\theta', \,d(\theta')=r^{^{-1/2}}}\|(f_k^{\not\sim})^{\ZT_{S_t,\theta'}}\|_{2}^2.
\end{equation}

Each $r$-tube $T'\in\ZT_{S_t}$ is associated to an $R/r^{1/2}\times R/r^{1/2}\times R$-tube $\wt T\supset T'$ that has the same direction as $T'$.
Note that only those scale-$R$ wave packets $Ef_T$ with $T\subset \wt T$ make contribution to the scale-$r$ wave packet $Ef_{T'}$ (see Lemma 7.1 in \cite{Guth-II}). Denote $\wt\cT$ as the collection of the fat tubes so that each $\wt T\in\wt \cT$ is associated to at least one $r$-tube $T'\in\bigcup_{S_t\in\cs_t}\ZT_{S_t}$ such that $Ef_{T'}\not=0$. Suppose $R/r^{1/2}\leq\rho_j$ but $R/r^{1/2}\geq\rho_{j-1}$ for some scale $\rho_j$ in \eqref{scales-1}. Then from the wave packet pruning in Section 2.1 we know that $\wt\cT$ contains $\lesssim R^{O(\de)} \ka_2(j)$ parallel tubes. Recall the broom estimate \eqref{broom-esti-2} 
\begin{equation}
\label{estimate-b}
    \|(f_k^{\not\sim})^{\ZT_{S_t,\theta'}}\|_{2}^2\lesssim R^{O(\de)} \ka_2(j)^{-1}r^{1/2}R^{-1/2}\|f_{\theta'}\|_{2}^2.
\end{equation}

A crucial observation here is that there is a Kakeya type constraint between $\si_1$ and $\ka_2(j)$. Indeed, suppose further that each $R/r^{1/2}$-ball in $\wt T$ contains at least one fat surface $S_t\in\cs_t$. Then each $R/r^{1/2}$-ball in the set $\bigcup_{\wt T\in\wt\cT}\wt T$ intersects $\gtrsim\si_1$ fat tubes in $\wt\cT$. Hence by Wolff's 5/2-maximal Kakeya estimate and the triangle inequality we have\footnote{There is in fact a stronger ``X-ray" estimate when $\ka_2(j)$ is big. But it is not useful to us here since $\ka_2(j)$ can be as small as $1$.}
\begin{equation}
\label{wolff-kakeya-1}
    \si_1\ka_2(j)^{-1}\lessapprox r^{1/4}.
\end{equation}
Plugging this back to \eqref{estimate-a} and \eqref{estimate-b} one gets a refinement
\begin{equation}
\label{estimate-c}
    \|(f_k^{\not\sim})^{\ZT_{S_t}}\|_{2}^2\lesssim R^{O(\de)} r^{3/4}R^{-1/2}\|f_{\theta'}\|_2^2,
\end{equation}
which is stronger than $\|(f_k^{\not\sim})^{\ZT_{S_t}}\|_{2}^2\lessapprox rR^{-1/2}\|f_{\theta'}\|_2^2$ that only uses the polynomial Wolff axiom (i.e. $\si_1\lesssim r^{1/2}$) and the broom estimate.

However, it is not always true that  each $R/r^{1/2}$-ball in $\wt T$ contains a fat surface $S_t\in\cs_t$. To deal with this issue,  define a shading $Y(\wt T)\subset\wt T$ as the union of $R/r^{1/2}$-balls in $\wt T$ that contains at least an $S_t\in\cs_t$ (see Definition \ref{shading-def-2}).  After pigeonholing and possibly refining $\wt \cT$, we may assume $|Y(\wt T)|\sim\la|\wt T|$ for some uniform constant $\la\leq1$ for all fat tubes $\wt T\in\wt\cT$. Again,  by Wolff's 5/2-maximal Kakeya estimate and the triangle inequality we have
\begin{equation}
\label{before-two-ends}
    \si_1\ka_2(j)^{-1}\lessapprox \la^{-3/2}r^{1/4}.
\end{equation}
The above estimate can be strengthened into, by Lemma \ref{two-ends-hairbrush},
\begin{equation}
\label{after-two-ends}
    \si_1\ka_2(j)^{-1}\lessapprox \la^{-1}r^{1/4},
\end{equation}
if the shading $Y(\wt T)$ is quantitative two-ends (see Definition \ref{two-ends-def}). Plugging this back to \eqref{estimate-a} and \eqref{estimate-b} one gets a refinement
\begin{equation}
\label{estimate-d}
    \|(f_k^{\not\sim})^{\ZT_{S_t}}\|_{2}^2\lesssim R^{O(\de)} \la^{-1}r^{3/4}R^{-1/2}\|f_{\theta'}\|_2^2,
\end{equation}
which is stronger than $\|(f_k^{\not\sim})^{\ZT_{S_t}}\|_{2}^2\lessapprox rR^{-1/2}\|f_{\theta'}\|_2^2$ when $\la\geq r^{-1/4}$.

\smallskip

The second method deals with the case $r\geq R^{2/3}$. Unlike the first method, in the second method we focus on $R$-tubes and $R^{1/2}$-balls. While it is essentially the same reason why we need the quantitative two-ends reduction---it strengthens some Kakeya estimate (see \eqref{before-two-ends} and \eqref{after-two-ends}).

\subsection{Sorting for leaves}
\label{sorting-section}

Recall in Lemma \ref{algebraic-lem-1} that there are a collection of leaves $\co_{leaf}$, and $n$ collections of fat surfaces $\{\cs_t\}_{1\leq t\leq n}$, each of which corresponds to a scale $r_t:=R_{j_t}$. In this subsection we would like to sort and find a $\gtrapprox1$ fraction of $O'\in\co_{leaf}$ that are distributed regularly in $S_t\in\cs_t$, $r_t\geq R^{1/2}$. This can be considered as a preparation for the two-ends reduction.

Let $1\leq m_1\leq m_2$ be two natural numbers with $r_{m_1}\geq R^{2/3}$, $r_{m_2}\geq R^{1/2}$ and $r_{m_1+1}\leq R^{2/3}$, $r_{m_2+1}\leq R^{1/2}$. We will first sort the leaves from the biggest scale $r_1$ to scale $r_{m_1}$, and then from scale $r_{m_1+1}$ to scale $r_{m_2}$.

\subsubsection{First sorting}
\label{first-sorting}

The sorting deals with the case $r_t\geq R^{2/3}$ and will be given in an iterative manner. Let $r_0=R$ and let $\co_{leaf,0}=\co_{leaf}$. Starting from scale $r_1$, suppose we have obtained a collection of leaves $\co_{leaf,t-1}$ at scale $r_{t-1}$. Now we sort $\co_{leaf,t-1}$ with respect to $\cs_t$ to obtain a refinement $\co_{leaf,t}\subset\co_{leaf,t-1}$ with $|\co_{leaf,t}|\gtrapprox|\co_{leaf,t-1}|$.

By pigeonholing, choose a set $\mathbf{q}_t$ of $r_t^{1/2}$--cubes such that 
\begin{enumerate}
\item  each $q\in \mathbf{q}_t$ contains about the same number of leaves in $\mathcal{O}_{leaf, t-1}$,
\item $\bigcup_{q\in\mathbf{q}_t} q$ contains at least a $ (\log R)^{-1}$-fraction of leaves in $\mathcal{O}_{leaf, t-1}.$
\end{enumerate}
Let $\cq$ denote a  set of  finitely overlapping $R^{1/2}$--cubes covering $B_R$. Since each $q\in\bq_t$ contains about the same amount of leaves in $\mathcal{O}_{leaf, t-1}$, by pigeonholing again, we can discard some cubes in $\bq_t$ to have either $S_t\cap Q=\varnothing$ or $S_t\cap Q$ contains about the same number of cubes $q\in\bq_t$ for all nonempty $S_t\cap Q$. Moreover, we still have that $\bigcup_{q\in\mathbf{q}_t} q$ contains at least a $ (\log R)^{-O(1)}$-fraction of leaves in $\mathcal{O}_{leaf, t-1}$.

Now there exists an injection (up to a constant factor)  
\begin{equation}
\label{injection}
    \mathbf{q}_t \rightarrow \mathbf{q}_t\times \mathcal{S}_t\times \cq: q\mapsto (q, S_t, Q)\text{ where } q\subset S_t\cap Q. 
\end{equation}
Each triple $(q, S_t, Q)$ is associated with a unique  triple of parameters $(\lambda_2, \lambda_3, \lambda_6)$ depending implicitly on $t$, where 
\begin{enumerate}
\item $\lambda_2$ means the number of $q'\in \mathbf{q}_t$ such that $q'\subset S_t\cap Q$. Write $\mathbf{q}_{S_t, Q}:=\{ q'\in \mathbf{q}_t: q'\subset S_t\cap Q\}$, then 
    \begin{equation}
    \label{la-2-t}
        |\bq_{S_t,Q}|\sim \la_2.
    \end{equation}
\item $\lambda_6$ means the number of $Q'\in \cq$ such that $|\mathbf{q}_{S_t, Q'}|\sim \lambda_2$ is about $\lambda_6$. For a fixed $S_t$, let $\mathcal{Q}_{S_t}$ denote the set of such $Q'$, then 
   \begin{equation}
    \label{la-6-t}
        |\cq_{S_t}|\sim\la_6. 
    \end{equation}
\item $\lambda_3$ means the number of $S_t'\in \mathcal{S}_t$ such that $|\mathbf{q}_{S_t', Q}|\sim \lambda_2$. For a fixed $Q$, let $\mathcal{S}_t(Q)$ denote the set of such $S_t'$, then 
  \begin{equation}
    \label{la-3-t}
        |\mathcal{S}_t(Q)|\sim\la_3. 
    \end{equation}
\end{enumerate}
We remark that the number $\la_2$ is uniform for all triples $(q,S_t,Q)$. This follows from the definition of $\bq_t$ (see above \eqref{injection}).

Since each $q$ is associated with a unique triple $(q, S_t, Q)$, it is also associated with a unique triple of parameters $(\lambda_2, \lambda_3, \lambda_6)$. By pigeonholing, there exists a uniform triple $(\lambda_2, \lambda_3, \lambda_6)$ such that the $q$'s that are associated with it consist of at least a $(\log R)^{-3}$-fraction of the original set $\mathbf{q}_t$. Note that if $(q, S_t, Q)$ is chosen, so is $(q', S_t, Q)$ for other $q'\in \mathbf{q}_t$ and $q'\subset S_t\cap Q$, namely, $S_t\cap Q$ is considered as a whole when doing pigeonholing. Denote by
\begin{enumerate}
    \item $\cs_t'$ the collection of fat surfaces $S_t$ that $|\cq_{S_t}|\sim\la_6$,
    \item $\bq_{S_t}$ the set of $q\in \mathbf{q}_t$ contained in $S_t$,
    \item $\cq'$ the collection of $R^{1/2}$-cubes $Q$ that $|\cs_t(Q)|\sim \la_3$.
\end{enumerate}
Then we have  $|\bq_{S_t}|\sim\la_2\la_6$ and
\begin{equation}
\label{pigeonholing-eq1}
    \la_2\la_3|\cq'|\approx|\bq_t|\approx \la_2\la_6|\cs_t'|.
\end{equation}
Let $\bq_t'=\{q\in\bq_t:q\in\bigcup_{S_t\in\cs_t'}\bq_{S_t}\}$, so $|\bq_t'|\gtrapprox|\bq_t|$. By \eqref{pigeonholing-eq1} and pigeonholing again there is a $\la_3'$ and a collection of $R^{1/2}$-balls $\cq_{\la_3'}$ so that $|\cs_t'(Q)|\sim\la_3'$ for any $Q\in\cq_{\la_3'}$, and the set $\{q\in\bq_t':q\subset \bigcup_{Q\in\cq_{\la_3'}}Q\}$ contains a fraction $\gtrapprox1$ of $r^{1/2}_t$-cubes in $\bq_t'$. To ease notations, still denote this fraction of $r^{1/2}_t$-cubes by $\bq_t$, and denote $\cs_t$ by $\cs_t'$, $\la_3$ by $\la_3'$. Hence we have
\begin{equation}
\label{pigeonholing-eq2}
    \la_3|\cq_{\la_3}|\approx\la_6|\cs_t|.
\end{equation}

Let $\mathcal{O}_{leaf, t}$ denote the set of leaves in $\mathcal{O}_{leaf, t-1}$ that each $O\in\co_{leaf,t}$ is contained in some $q\in \mathbf{q}_t$.  
Since each $q\in \mathbf{q}_t$ contains about the same number of leaves in $\mathcal{O}_{leaf, t-1}$, we get 
\[
|\mathcal{O}_{leaf, t}|\gtrapprox |\mathcal{O}_{leaf, t-1}|. 
\]
Since $S_t$ is supported in $N_{r_t^{1/2+\de}}(Z_{S_t})\cap B_{S_t}$ for some variety $Z_{S_t}$ with degree at most $d$ and some ball $B_{S_t}$ of radius $r_t$, by Wongkew's theorem (see \cite{Guth-R3} Theorem 4.7) one can bound the Lebesgue measure of $S_t$ from above as $|S_t|\lesssim dr_t^{5/2}$, yielding
\begin{equation}
\label{bq-S-t}
    |\bq_{S_t}|\lessapprox r_t.
\end{equation}

\begin{figure}	
\begin{tikzpicture}
 \begin{scope}
   [shift={(-1.5,-1.5)},x={(2.5cm,0cm)},
    y={({cos(60)*1.5cm},{sin(60)*1.5cm})},
    z={({cos(90)*2.5cm},{sin(90)*2.5cm})},line join=round,fill opacity=0.5, blue]
  \draw[] (0,0,0) -- (0,0,1) -- (0,1,1) -- (0,1,0) -- cycle;
  \draw[] (0,0,0) -- (1,0,0) -- (1,1,0) -- (0,1,0) -- cycle;
  \draw[] (0,1,0) -- (1,1,0) -- (1,1,1) -- (0,1,1) -- cycle;
  \draw[] (1,0,0) -- (1,0,1) -- (1,1,1) -- (1,1,0) -- cycle;
  \draw[] (0,0,1) -- (1,0,1) -- (1,1,1) -- (0,1,1) -- cycle;
  \draw[] (0,0,0) -- (1,0,0) -- (1,0,1) -- (0,0,1) -- cycle;
 \end{scope}
  \begin{scope}
   [shift={(0,-.9)},x={(.5cm,0cm)},
    y={({cos(60)*.3cm},{sin(60)*.3cm})},
    z={({cos(90)*.5cm},{sin(90)*.5cm})},line join=round,fill opacity=0.5, dotted, black]
  \draw[fill=red] (0,0,0) -- (0,0,1) -- (0,1,1) -- (0,1,0) -- cycle;
  \draw[fill=red] (0,0,0) -- (1,0,0) -- (1,1,0) -- (0,1,0) -- cycle;
  \draw[fill=red] (0,1,0) -- (1,1,0) -- (1,1,1) -- (0,1,1) -- cycle;
  \draw[fill=red] (1,0,0) -- (1,0,1) -- (1,1,1) -- (1,1,0) -- cycle;
  \draw[fill=red] (0,0,1) -- (1,0,1) -- (1,1,1) -- (0,1,1) -- cycle;
  \draw[fill=red] (0,0,0) -- (1,0,0) -- (1,0,1) -- (0,0,1) -- cycle;
 \end{scope}
 
 \begin{scope}
   [rotate=10,shift={(.3,0)},x={(.5cm,0cm)},
    y={({cos(60)*.3cm},{sin(60)*.3cm})},
    z={({cos(90)*.5cm},{sin(90)*.5cm})},line join=round,fill opacity=0.5, dotted, black]
  \draw[fill=red] (0,0,0) -- (0,0,1) -- (0,1,1) -- (0,1,0) -- cycle;
  \draw[fill=red] (0,0,0) -- (1,0,0) -- (1,1,0) -- (0,1,0) -- cycle;
  \draw[fill=red] (0,1,0) -- (1,1,0) -- (1,1,1) -- (0,1,1) -- cycle;
  \draw[fill=red] (1,0,0) -- (1,0,1) -- (1,1,1) -- (1,1,0) -- cycle;
  \draw[fill=red] (0,0,1) -- (1,0,1) -- (1,1,1) -- (0,1,1) -- cycle;
  \draw[fill=red] (0,0,0) -- (1,0,0) -- (1,0,1) -- (0,0,1) -- cycle;
 \end{scope}

  \begin{scope}
   [shift={(-1,-.9)},x={(8cm,0cm)},
    y={({cos(60)*5cm},{sin(60)*5cm})},
    z={({cos(90)*.5cm},{sin(90)*.5cm})},line join=round,fill opacity=0.5, thick, dotted, red]
  \draw[] (0,0,0) -- (0,0,1) -- (0,1,1) -- (0,1,0) -- cycle;
  \draw[] (0,0,0) -- (1,0,0) -- (1,1,0) -- (0,1,0) -- cycle;
  \draw[] (0,1,0) -- (1,1,0) -- (1,1,1) -- (0,1,1) -- cycle;
  \draw[] (1,0,0) -- (1,0,1) -- (1,1,1) -- (1,1,0) -- cycle;
  \draw[] (0,0,1) -- (1,0,1) -- (1,1,1) -- (0,1,1) -- cycle;
  \draw[] (0,0,0) -- (1,0,0) -- (1,0,1) -- (0,0,1) -- cycle;
 \end{scope}
   \begin{scope}
   [rotate=10,shift={(-1,0)},x={(8cm,0cm)},
    y={({cos(60)*5cm},{sin(60)*5cm})},
    z={({cos(90)*.5cm},{sin(90)*.5cm})},line join=round,fill opacity=0.5, thick, dotted, red]
  \draw[] (0,0,0) -- (0,0,1) -- (0,1,1) -- (0,1,0) -- cycle;
  \draw[] (0,0,0) -- (1,0,0) -- (1,1,0) -- (0,1,0) -- cycle;
  \draw[] (0,1,0) -- (1,1,0) -- (1,1,1) -- (0,1,1) -- cycle;
  \draw[] (1,0,0) -- (1,0,1) -- (1,1,1) -- (1,1,0) -- cycle;
  \draw[] (0,0,1) -- (1,0,1) -- (1,1,1) -- (0,1,1) -- cycle;
  \draw[] (0,0,0) -- (1,0,0) -- (1,0,1) -- (0,0,1) -- cycle;
 \end{scope}

\node at (6,0.3) {$S_t$};
\node at (6,5) {$S_t'$};

\node at (.55,0.45) {$q'$};
\node at (.3,-.7) {$q$};

\node at (-.5,-1.2) {$Q$};

\node at (3,-2) {$S_t,S_t'$: fat $r^{1/2}$-surfaces; \,\, $Q$: $R^{1/2}$-ball; \,\, $q,q'$: $r^{1/2}$-cubes.};

\node at (3,-2.5) {Every leaf $O'$ is contained in some $r^{1/2}$-ball $q$.};

\end{tikzpicture}
\caption{Relations between different geometric objects}

\label{geometric-object}
\end{figure}

This finishes the sorting for leaves at the scale $r_t$. We remark that now each $Q\in\cq_{\la_3}$ contains about the same amount of leaves in $\co_{leaf,t}$.

\begin{lemma}\label{lem: conservation of lambdas}
Let  $r_t\geq R^{2/3}$ and $\co_{leaf}\subset \co_{leaf,t}$ be a subset with $|\co_{leaf}|\gtrapprox R^{-\de/2}|\co_{leaf,t}|$. Then there exists subsets $\mathbf{q}_t'\subset \mathbf{q}_t, \mathcal{S}_t'\subset \mathcal{S}_t$ and parameters $\lambda_2', \lambda_3', \lambda_6'$ satisfying Subsection~\ref{first-sorting} with $\lambda_j'$ in the place of $\lambda_j$, $j=2, 3, 6$ and $|\lambda_j'|\gtrsim R^{-4\delta} |\lambda_j|$. In addition, each $q\in \mathbf{q}_t'$ contains about the same number of leaves in $\co_{leaf}$  (up to a factor of $R^{\delta}$) and $|\co_{leaf}'|\gtrapprox|\co_{leaf}|$ where $\co_{leaf}'\subset \co_{leaf}$ is the set of leaves that each of which is contained in some $q\in\bq_t'$. 
\end{lemma}
\begin{proof}

Since $|\co_{leaf}|\gtrapprox R^{-\de/2}|\co_{leaf, t}|$, and each $q\in \mathbf{q}_t$ in Subsection~\ref{first-sorting} contains about the same number of leaves in $\co_{leaf, t}$, we can discard the $r_t^{1/2}$--cubes from $\mathbf{q}_t$  that contains a less than  $R^{-\delta}$-fraction of the original cells. 

Now each $q\in \mathbf{q}_t$ contains about the same number of leaves in $\co_{leaf}$ and $|\co_{leaf}'|\gtrapprox |\co_{leaf}|$ where $\co_{leaf}'\subset \co_{leaf}$ is the set of leaves in $q\in \mathbf{q}_t$.

Proceed as in Subsection~\ref{first-sorting} and find a subset $\mathbf{q}_t'\subset \mathbf{q}_t$ satisfying same
the uniform property and contains a significant fraction of the leaves and for each $q\in \mathbf{q}_t'$, $q$ is uniquely associated with a triple $(q, S_t, Q)\in \mathbf{q}_t'\times \mathcal{S}_t'\times \mathcal{Q}'$ and a triple of parameters $(\lambda_2', \lambda_3', \lambda_6')$, where $\cs_t'\subset\cs_t$ and $\cq'\subset\cq$. If $\lambda_2' \leq R^{-4\delta} \lambda_2$, then $|\co_{leaf}'|\lessapprox R^{-2\delta} |\co_{leaf,t}|$, which is a contradiction. Same reason applies to $\lambda_3,\lambda_6$. 
\end{proof}

\subsubsection{Second sorting}
\label{second-sorting}
The second sorting will also be given in an iterative manner, and is simpler than the first one. From the first sorting we know that there is a collection of leaves $\co_{leaf,m_1}$ at scale $r_{m_1}$. Starting from the scale $r_{m_1+1}$, suppose we have obtained a collection of leaves $\co_{leaf,t-1}$ from scale $r_{t-1}$. Now we sort $\co_{leaf,t-1}$ at scale $r_t$ to obtain a refinement $\co_{leaf,t}\subset\co_{leaf,t-1}$ with $|\co_{leaf,t}|\gtrapprox|\co_{leaf,t-1}|$.

By pigeonholing, we can find $\cb(t)$, a collection finitely overlapping $R/r_t^{1/2}$-balls in $B_R$, so that each $B\in\cb(t)$ contains about the same amount of leaves in $\co_{leaf,t-1}$ up to a constant multiple, and the set $\co_{leaf,t}:=\{O'\in\co_{leaf,t-1}:O'\subset B\text{ for some }B\in\cb(t)\}$ satisfies $|\co_{leaf,t}|\gtrapprox|\co_{leaf,t-1}|$. This finishes the sorting for leaves at the scale $r_t$.

\medskip

Finally, to ease the notation we set $\co_{leaf}=\co_{leaf,m_2}$. This finishes our sorting of leaves.

\subsection{The quantitative two-ends reductions}

We will realize the reductions by defining a new relation $\sim_{\bn}$ (the relation $\sim$ mentioned in Section \ref{subsection-broom} is recognized as the old relation) between the $R$-tubes $T$ and the $R^{1-\e_0}$-balls $B_k$ (recall $\e_0=\e^{10}$, and see Section \ref{subsection-broom} for $B_k$). Recall from the previous subsection (see also Lemma \ref{algebraic-lem-1}) that there are $n$ collections of fat surfaces $\{\cs_t\}_{1\leq t\leq n}$, each of which corresponds to a scale $r_t:=R_{j_t}$, and there are two natural numbers $1\leq m_1\leq m_2$ that $r_{m_1}\geq R^{2/3}$, $r_{m_2}\geq R^{1/2}$ while $r_{m_1+1}\leq R^{2/3}$, $r_{m_2}\leq R^{1/2}$.

\subsubsection{First type of related tubes}
\label{first-related-tubes}
Suppose at first $1\leq t\leq m_1$. For each $R^{1-\e_0}$-ball $B_{k}$, let us define a collection of related tubes, which is a subset of $\ZT_{\la_1}[R]$ (see \eqref{la-1} for $\ZT_{\la_1}[R]$). To do so, we would like to define a shading $Y(T)\subset T$ for each $T\in\ZT_{\la_1}[R]$. Recall that in the first sorting (Section \ref{first-sorting}) there is a collection of $R^{1/2}$-balls $\cq_{\la_3}=\cq_{\la_3}(t)$.

\begin{definition}
\label{shading-def-1}
For each $R$-tube $T\in\ZT_{\la_1}[R]$, define a shading $Y(T)$ as the union of all the $R^{1/2}$-balls $Q\subset T$ that $Q\in\cq_{\la_3}$. Namely, 
\begin{equation}
\label{shading-1}
    Y(T)=\bigcup_{Q\in\cq_{\la_3}}T\cap Q.
\end{equation}
\end{definition}
For each $T\in\ZT_{\la_1}[R]$, partition it into $R^{\e_0}$ many $R^{1-\e_0}$-segments $\{T_j\}$. We sort the segments according to $|Y(T_j)|$. That is, let $J_\al^t=J_\al^t(T)$ be the collection of $j$ such that $|Y(T_j)|\sim \al$. Define a new shading ($Y_\al$ implicitly depends on $t$)
\begin{equation}
\label{new-shading-1}
    Y_\al(T)=\bigcup_{j\in J_\al^t}Y(T_j).
\end{equation}
For each $R^{1-\e_0}$-ball $B_{k}$, define $\ZT_{k}=\{T\in\ZT_{\la_1}[R]: Y(T)\cap B_{k}\not=\varnothing\}$. Then we can partition $\ZT_k$ as
\begin{equation}
\label{alpha-lambda-4}
    \ZT_{k}=\bigsqcup_{\al,\la}\ZT_{k,\al,\la}(t)
\end{equation}
so that for all $T\in\ZT_{k,\al,\la}$
\begin{equation}
\label{la-4-t}
    |Y_\al(T)|\sim \la|T|,
\end{equation}
and hence $|J_\al^t(T)|$ are about the same up to a constant multiple. We remark that for   $Q\in\cq_{\la_3}$, $Q\cap Y_\al(T)=Q\cap T$ when $T\in\ZT_{k,\al,\la}$ and $Q\subset B_{k}$

Recall $\e_0'=\e^{15}$. We will distinguish two cases: $|J_\al^t(T)|\leq R^{\e_0'}$ and $|J_\al^t(T)|\geq R^{\e_0'}$. For each $R^{1-\e_0}$-ball $B_{k}$, define the collection of related tubes and non-related tubes ($\sim_{\bn}$ stands for ``new relation")
\begin{equation}
\label{scale-t-related}
    \ZT^{\sim_{\bn}}_{k}(t)=\{T\in\ZT_{k}:|J_\al^t(T)|\leq R^{\e_0'} \}, \hspace{.5cm}\ZT^{\not\sim_{\bn}}_{k}(t)=\ZT_{k}\setminus\ZT^{\sim_{\bn}}_{k}(t).
\end{equation}
Hence any related tube $T\in\bigcup_{k}\ZT^{\sim_{\bn}}_{k}(t)$ belongs to $\lesssim R^{\e_0'}$ sets $\{\ZT^{\sim_{\bn}}_{k}(t)\}_k$, and thus is related to $\lesssim R^{\e_0'}$ balls $B_{k}$. Recall \eqref{alpha-lambda-4} and notice that for a fixed pair $(\al,\la)$, either $\ZT_{k,\al,\la}(t)\subset\ZT^{\sim}_{k}(t)$ or $\ZT_{k,\al,\la}(t)\subset\ZT^{\not\sim_{\bn}}_{k}(t)$. Define 
\begin{equation}
\label{non-al-la-1}
    \ZT^{\not\sim_{\bn}}_{k,\al,\la}(t)=\ZT_{k,\al,\la}(t)\cap \ZT^{\not\sim_{\bn}}_{k}(t).
\end{equation}
Since $|J_\al^t(T)|\geq R^{\e_0'}$ for any $T\in\bigcup_{k}\ZT^{\not\sim_{\bn}}_{k,\al,\la}(t)$, we have that $Y_\al(T)$ is quantitative two-ends at scale $\e_0'$ (see Definition \ref{two-ends-def}). What follows is a useful incidence estimate among the $R^{1/2}$-balls in $\cq_{\la_3}$ and the $R$-tubes in $\ZT^{\not\sim_{\bn}}_{k,\al,\la}(t)$.

\begin{lemma}
\label{kakeya-type-lem-1}
Fix a pair $(\al,\la)$. Denote by $\nu(Q)$ the number of shading $\{Y_\al(T):T\in\bigcup_{k}\ZT^{\not\sim_{\bn}}_{k,\al,\la}\}$ intersects $Q$. Then, recalling $|Y_\al(T)|\sim \la|T|$ and $\Theta_{\la_1}[R]$ near \eqref{la-1}, we clearly have
\begin{equation}
\label{l1-estimate}
    \sum_{Q\in\cq_{\la_3}}R^{3/2}\nu(Q)=\sum_{T\in\bigcup_{k}\ZT^{\not\sim_{\bn}}_{k,\al,\la}}|Y_{\al}(T)|\lesssim\la\la_1R^{2}|\Theta_{\la_1}[R]|.
\end{equation}
More importantly, there is a subset $\cq_{\la_3,\la}\subset\cq_{\la_3}$ with
\begin{equation}
\label{cq-la3-la4-t}
    |\cq_{\la_3}\setminus\cq_{\la_3,\la}|\lessapprox R^{-\Om(\de)}|\cq_{\la_3}|
\end{equation}
so that $\nu(Q)\lesssim R^{O(\de)}R^{1/4}\la^{-1}\la_1$ whenever $Q\in\cq_{\la_3,\la}$. 
\end{lemma}
\begin{proof}
Note that by \eqref{la-1}, $ \bigcup_{k}\ZT^{\not\sim_{\bn}}_{k,\al,\la}$ contains $\lesssim\la_1$ parallel tubes. After rescaling, apply Lemma \ref{two-ends-hairbrush} with $\cT= \bigcup_{k}\ZT^{\not\sim_{\bn}}_{k,\al,\la}$, $m=\la_1$, $Y=Y_\al$, and $A=\bigcup_{T\in\cT}Y_\al(T)$ (which is a subset of $\bigcup_{Q\in\cq_{\la_3}}Q$). Thus, there is a dyadic number
\begin{equation}
    \mu\lessapprox R^{1/4}\la^{-1}\la_1
\end{equation}
and a set $A_\mu\subset A$, such that for any other dyadic number $\mu'$, the set $A_{\mu'}$ satisfies $|A_{\mu'}|\lesssim\mu|A_{\mu}|/\mu'$. Summing up all dyadic $\mu'\geq R^{O(\de)}\mu$ one has
\begin{equation}
\label{kakeya-implication-1-1}
    \Big|\bigcup_{\mu'\gtrapprox  R^{O(\de)}\mu} A_{\mu'}\Big|\lessapprox R^{-\Om(\de)}|A_\mu|\lessapprox R^{-\Om(\de)}|A|.
\end{equation}
Since $\bigcup_{\mu'\gtrapprox  R^{O(\de)}\mu}A_{\mu'}$ is a subset of $\bigcup_{Q\in\cq_{\la_3}}Q$, there is a subset $\cq_{\la_3,\la}\subset\cq_{\la_3}$ with
\begin{equation}
    |\cq_{\la_3}\setminus\cq_{\la_3,\la}|\lessapprox R^{-\Om(\de)}|\cq_{\la_3}|,
\end{equation}
so that each $Q\in\cq_{\la_3,\la}$ intersects $\nu(Q)\lesssim R^{O(\de)}R^{1/4}\la^{-1}\la_1$ shadings $Y_\al(T)$.
\end{proof}

Since for each $R^{1-\e_0}$-ball $B_{k}$ and each $T\in\ZT^{\not\sim_{\bn}}_{k,\al,\la}$ we have for $Q\in\cq_{\la_3}$, $Q\cap Y_\al(T)=Q\cap T$ when $Q\subset B_{k}$. A direct corollary of Lemma \ref{kakeya-type-lem-1} is the following.
\begin{corollary}
\label{corollary-kakeya-1}
Fix a pair $(\al,\la)$. For each $R^{1-\e_0}$-ball $B_{k}$, let $\cq_{\la_3}^{k}=\{Q\in\cq_{\la_3}:Q\cap B_{k}\not=\varnothing\}$. Suppose $|\cq_{\la_3}^{k}|\geq R^{-O(\de^2)} |\cq_{\la_3}|$. Then there is a subset $\cq_{\la_3,\la}^{k}\subset\cq_{\la_3}^{k}\cap\cq_{\la_3,\la}$ with
\begin{equation}
\label{cq-la3-la4-t-cor}
    |\cq_{\la_3}^{k}\setminus\cq_{\la_3,\la}^{k}|\lessapprox R^{-\Om(\de)}|\cq_{\la_3}^{k}|
\end{equation}
so that each $Q\in\cq_{\la_3,\la}^{k}$ intersects $\lesssim R^{O(\de)}R^{1/4}\la^{-1}\la_1$ tubes $T\in\ZT^{\not\sim_{\bn}}_{k,\al,\la}$.
\end{corollary}

\subsubsection{Second type of related tubes}
\label{second-related-tubes}
Suppose $m_1+1\leq t\leq m_2$. In this case, we focus on those $R/r_t^{1/2}$-balls in $\cb(t)$ (see Section \ref{second-sorting}, the second sorting of leaves) and those fat tubes $\wt T\in\wt\cT(t)$, where $\wt \cT(t)$ is a collection of $R/r_t^{1/2}\times R/r_t^{1/2}\times R$-tubes that each $\wt T\in\wt\cT(t)$ contains at least one $R$-tube $T$ that $(f_{\rho_l})_T\not=0$ (see Section \ref{wp-pruning-section} for the step-$\rho_l$ function $f_{\rho_l}$), and that any two fat tubes in $\wt\cT(t)$ either are parallel, or make an angle $\gtrsim r_t^{-1/2}$. Let $\rho_j$ be that $\rho_j\geq R/r_t^{1/2}\geq \rho_{j-1}$ (see \eqref{scales-1} for $\rho_j$), so there are at most $R^{O(\de)}\ka_2(j)$ parallel tubes in $\wt\cT(t)$ (see above \eqref{l2-estimate-pruning}).

Similar to Definition \ref{shading-def-1}, we define a shading $Y(\wt T)\subset \wt T$ on each fat tube $\wt T\in\wt\cT(t)$ via the $R/r^{1/2}$-balls in $\cb(t)$.

\begin{definition}
\label{shading-def-2}
For each fat tube $\wt T\in\wt\cT(t)$, define a shading $Y(\wt T)$ as the union of all the $R/r_t^{1/2}$-balls in $\cb(t)$ that also intersect $\wt T$. Namely, 
\begin{equation}
\label{shading-2}
    Y(\wt T)=\bigcup_{B\in\cb(t)}\wt T\cap B.
\end{equation}
\end{definition}

Recall $\e_0=\e^{10}$ and $\e_0'=\e^{15}$. For each $\wt T\in\wt\cT(t)$, we similarly partition it into $R^{\e_0}$ many $R^{1-\e_0}$-segments $\{\wt T_j\}$. We sort the segments according to $|Y(\wt T_j)|$: Let $J_\al^t$ be the collection of $j$ such that $|Y(\wt T_j)|\sim \al$. Define a new shading
\begin{equation}
\label{new-shading-2}
    Y_\al(\wt T)=\bigcup_{j\in J_\al^t}Y(\wt T_j).
\end{equation}
For each $R^{1-\e_0}$-ball $B_{k}$, define $\wt\cT_{k}(t)=\{\wt T\in\wt\cT(t): Y(\wt T)\cap B_{k}\not=\varnothing\}$. Then we can partition it as
\begin{equation}
\label{alpha-lambda}
    \wt\cT_{k}(t)=\bigsqcup_{\al,\la}\wt\cT_{k,\al,\la}(t)
\end{equation}
so that for all $\wt T\in\wt\cT_{k,\al,\la}(t)$
\begin{equation}
\label{la-t}
    |Y_\al(\wt T)|\sim \la|\wt T|,
\end{equation}
and hence the mass $|J_\al^t(\wt T)|$ are about the same up to a constant multiple. Note that for $B\in\cb(t)$, $B\cap Y_\al(\wt T)=B\cap \wt T$ when $\wt T\in\wt\cT_{k}(t)$ and $B\subset B_{k}$

We similarly distinguish two cases: $|J_\al^t(\wt T)|\leq R^{\e_0'}$ and $|J_\al^t(\wt T)|\geq R^{\e_0'}$. For each $R^{1-\e_0}$-ball $B_{k}$, define the collection of related tubes and non-related tubes
\begin{equation}
\label{scale-t-related-fat}
    \wt\cT^{\sim_{\bn}}_{k}(t)=\{\wt T\in\wt\cT:|J_\al^t(\wt T)|\leq R^{\e_0'}\}, \hspace{.5cm}\wt\cT^{\not\sim_{\bn}}_{k}(t)=\wt\cT_{k}(t)\setminus\wt\cT^{\sim_{\bn}}_{k}(t).
\end{equation}
Hence any related tube $\wt T\in\wt\cT^{\sim_{\bn}}(t)$ is related to $\lesssim R^{\e_0'}$ balls $B_{k}$. Now for each $B_{k}$ define a collection of related $R$-tubes 
\begin{equation}
\label{scale-t-related-2}
    \ZT^{\sim_{\bn}}_{k}(t)=\{T\in\ZT_{\la_1}[R]:T\subset\wt T\text{ for some }\wt T\in\wt\cT^{\sim_{\bn}}_{k}(t)\}
\end{equation}
as well as a collection of non-related $R$-tubes
\begin{equation}
\label{scale-t-nonrelated-2}
    \ZT^{\not\sim_{\bn}}_{k}(t)=\{T\in\ZT_{\la_1}[R]:T\subset\wt T\text{ for some }\wt T\in\wt\cT^{\not\sim_{\bn}}_{k}(t)\}. 
\end{equation}

Recall \eqref{alpha-lambda}. Fix a pair $(\al,\la)$, define
\begin{equation}
\label{non-al-la-2}
    \wt\cT^{\not\sim_{\bn}}_{k,\al,\la}(t)=\wt\cT_{k,\al,\la}(t)\cap \wt\cT^{\not\sim_{\bn}}_{k}(t)
\end{equation}
as well as
\begin{equation}
\label{non-al-la-3}
    \ZT^{\not\sim_{\bn}}_{k,\al,\la}(t)=\{T\in\ZT_{\la_1}[R]:T\subset\wt T\text{ for some }\wt T\in\wt\cT^{\not\sim_{\bn}}_{k,\al,\la}(t)\}.
\end{equation}
Note that for each fat tube $\wt T\in\bigcup_{k}\wt\cT^{\not\sim_{\bn}}_{k,\al,\la}(t)$, by definition $|J_\al^t(\wt T)|\geq R^{\e_0'}$ and $|Y_\al(\wt T)|\sim \la|\wt T|$. What follows is an incidence estimate among the $R/r_t^{1/2}$-balls in $\cb(t)$ and the fat tubes in $\bigcup_{k}\wt\cT^{\not\sim_{\bn}}_{k,\al,\la}(t)$.

\begin{lemma}
\label{kakeya-type-lem-2}
Fix a pair $(\al,\la)$. There is a subset $\cb_\la(t)\subset\cb(t)$ with
\begin{equation}
\label{cb-m-al-la}
    |\cb(t)\setminus\cb_\la(t)|\lessapprox R^{-\Om(\de)}|\cb(t)|
\end{equation}
so that each $R/r_t^{1/2}$-ball $B\in\cb_\la(t)$ intersects $\lesssim R^{O(\de)}r_t^{1/4}\la^{-1}\ka_2(j)$ shadings $Y_\al(\wt T)$ with $\wt T\in\bigcup_{k}\wt\cT^{\not\sim_{\bn}}_{k,\al,\la}(t)$. 
\end{lemma}
The proof of Lemma \ref{kakeya-type-lem-2} is similar to the one of Lemma \ref{kakeya-type-lem-1}. One just needs to notice that $\bigcup_{k}\wt\cT^{\not\sim_{\bn}}_{k,\al,\la}(t)$ contains $\lesssim R^{O(\de)}\ka_2(j)$ parallel fat tubes (see the beginning of Section \ref{second-related-tubes}). We omit details. 

Similar to Corollary \ref{corollary-kakeya-1}, we have a corollary of Lemma \ref{kakeya-type-lem-2}.

\begin{corollary}
\label{corollary-kakeya-2}
Fix a pair $(\al,\la)$. For each $R^{1-\e_0}$-ball $B_{k}$, let $\cb^{k}(t)=\{B\in\cb(t):B\cap B_{k}\not=\varnothing\}$. Suppose $|\cb^{k}(t)|\geq R^{-O(\de^2)} |\cb(t)|$. Then there is a subset $\cb_\la^{k}(t)\subset\cb^{k}(t)\cap\cb_\la(t)$ with
\begin{equation}
\label{cb-m-al-la-cor}
    |\cb^{k}(t)\setminus\cb_\la^{k}(t)|\lessapprox R^{-\Om(\de)}|\cb^{k}(t)|
\end{equation}
so that each $R/r_t^{1/2}$-ball $B\in\cb_\la(t)$ intersects $\lesssim R^{O(\de)}r_t^{1/4}\la^{-1}\ka_2(j)$ fat tubes $\wt T\in\wt\cT^{\not\sim_{\bn}}_{k,\al,\la}(t)$ (see Section \ref{wp-pruning-section} for $\ka_2(j)$). Also, since $Y_\al(\wt T)\cap B_{k}=\wt T\cap B_{k}$, every $\wt T\in\wt\cT^{\not\sim_{\bn}}_{k,\al,\la}(t)$ intersects $\lesssim \la r_t^{1/2}$ balls in $\cb_\la^{k}(t)$.
\end{corollary}

\medskip

At this point we finish defining the new related tubes at all scales $r_1,\ldots,r_{m_2}$. Recall the old relation $\sim$ in Section \ref{subsection-broom}. Define for each $R^{1-\e_0}$-ball $B_{k}$ the ultimate related tubes and non-related tubes as (recall \eqref{related-wp} and \eqref{unrelated-wp})
\begin{equation}
\label{ultimate-related}
    \ZT^{\sim_{\bn}}_{k}=\ZT_{k}^{\sim}\cup\Big(\bigcup_{t=1}^{m_2}\ZT^{\sim_{\bn}}_{k}(t)\Big), \hspace{.5cm}\ZT^{\not\sim_{\bn}}_{k}=\Big(\bigcap_{t=1}^{m_2}\ZT^{\not\sim_{\bn}}_{k}(t)\Big)\cap\ZT_k^{\not\sim},
\end{equation}
so that each $T\in\bigcup_{k}\ZT^{\sim_{\bn}}_{k}$ belongs to $\lesssim R^{\e_0'}$ collections $\{\ZT^{\sim_{\bn}}_{k}\}_{k}$. From these we can define a related function and a non-related function
\begin{equation}
\label{ultimate-related-fcn}
    f^{\sim_{\bn}}_{k}=\Id_{B_{k}}\sum_{T\in\ZT^{\sim_{\bn}}_{k}}f_T, \hspace{.5cm}f^{\not\sim_{\bn}}_{k}=\Id_{B_{k}}\sum_{T\in\ZT^{\not\sim_{\bn}}_{k}}f_T=\Id_{B_{k}}f_k-f^{\sim_{\bn}}_{k},
\end{equation}
where $f_k=\sum_{T\in\ZT_k}f_T$ (see above \eqref{alpha-lambda-4} for $\ZT_k$). We remark that the unrelated function $f^{\not\sim_{\bn}}_{k}$ still enjoys the broom estimate \eqref{broom-esti-2} (see Remark \ref{broom-remark-0}).

\subsection{Handling the related function}
Recall Lemma \ref{algebraic-lem-1} that $\|Ef\|_{\BL^p(O')}^p$ are about the same for all $O'\in\co_{leaf}$. Suppose there is a fraction $\gtrapprox1$ of $O'\in\co_{leaf}$ that $\|Ef\|_{\BL^p(O')}^p\lesssim\sum_{k}\|Ef^{\sim_{\bn}}_{k}\|_{\BL^p(O')}^p$. Then we can conclude our main estimate \eqref{main-esti} by induction on scales.

\begin{lemma}
\label{induction-lem-1}
Suppose \eqref{main-esti} is true for scales $\leq R/2$. Then 
\begin{equation}
    \sum_{B_{k}}\|Ef^{\sim_{\bn}}_{k}\|_{L^p(B_{k})}^p\leq C_\e^p R^{(1-\e_0)p\e+2\e_0'}\|f\|_2^2\sup_{\theta}\|f_\theta\|_{L^2_{ave}}^{p-2}.
\end{equation}
Therefore, from \eqref{partitioning-blp-norm} and the fact that there is a fraction  $\gtrapprox1$ of $O'\in\co_{leaf}$ that $\|Ef\|_{\BL^p(O')}^p\lesssim\sum_{k}\|Ef^{\sim_{\bn}}_{k}\|_{\BL^p(O')}^p$, one has
\begin{align}
    \|Ef\|_{\BL^p(B_R)}^p\lessapprox &\,  C_\e^p R^{(1-\e_0)p\e+2\e_0'}\|f\|_2^2\sup_{\theta}\|f_\theta\|_{L^2_{ave}}^{p-2}\\
    \leq &\, R^{p\e}\|f\|_2^2\sup_{\theta}\|f_\theta\|_{L^2_{ave}}^{p-2},
\end{align}
which closes the induction and hence yields \eqref{main-esti}.
\end{lemma}

\begin{proof}

Using the induction hypothesis \eqref{main-esti} at  scale $R^{1-\e_0}$ we have
\begin{equation}
    \|Ef^{\sim_{\bn}}_{k}\|_{L^p(B_{k})}^p\leq C_\e^p R^{(1-\e_0)p\e}\|f^{\sim_{\bn}}_{k}\|_2^2\sup_{\theta}\|(f^{\sim_{\bn}}_{k})_\theta\|_{L^2_{ave}}^{p-2}.
\end{equation}
Note that each $T\in\bigcup_{k}\ZT^{\sim_{\bn}}_{k}$ belongs to $\lesssim R^{\e_0'}$ collections $\{\ZT^{\sim_{\bn}}_{k}\}_{k}$. Hence, after summing up all $B_{k}$ one has by Plancherel
\begin{align}
    &\sum_{B_{k}}\|Ef^{\sim_{\bn}}_{k}\|_{L^p(B_{k})}^p\leq C_\e^p R^{(1-\e_0)p\e}\sum_{B_{k}}\|f^{\sim_{\bn}}_{k}\|_2^2\sup_{\theta}\|(f^{\sim_{\bn}}_{k})_\theta\|_{L^2_{ave}}^{p-2}\\ \nonumber
    \leq &\, CC_\e^p R^{(1-\e_0)p\e+\e_0'}\|f\|_2^2\sup_{\theta}\|f_\theta\|_{L^2_{ave}}^{p-2}\leq C_\e^p R^{(1-\e_0)p\e+2\e_0'}\|f\|_2^2\sup_{\theta}\|f_\theta\|_{L^2_{ave}}^{p-2}.
\end{align}
This is what we desire.
\end{proof}

\medskip
Suppose instead there is a fraction $\gtrapprox1$ of $O'\in\co_{leaf}$ that $\|Ef\|_{\BL^p(O')}^p \lesssim\sum_{k}\|Ef^{\not\sim_{\bn}}_{k}\|_{\BL^p(O')}^p$. We will handle this case in the rest of the paper.

\section{Finding the correct scale}

From the end of last section, we know that there is a fraction $\gtrapprox1$ of $O'\in\co_{leaf}$ that $\|Ef\|_{\BL^p(O')}^p\lesssim\sum_{k}\|Ef^{\not\sim_{\bn}}_{k}\|_{\BL^p(O')}^p$. By pigeonholing, there is an $R^{1-\e_0}$-ball $B_{k}$ (which we fix from now on) and a subset $\co_{leaf,k}^{\not\sim}\subset\co_{leaf}$ so that
\begin{enumerate}
    \item $|\co_{leaf,k}^{\not\sim}|\gtrapprox R^{-10\e_0}|\co_{leaf}|$.
    \item $\|Ef^{\not\sim_{\bn}}_{k}\|_{\BL^p(O')}$ are about the same up to constant multiple for all $O'\in\co_{leaf,k}^{\not\sim}$.
    \item It holds that
    \begin{equation}
    \label{section-4-eq-1}
        \sum_{O'\in\co_{leaf}}\|Ef\|_{\BL^p(O')}^p\lessapprox R^{10\e_0}\sum_{O'\in\co_{leaf,k}^{\not\sim}}\|Ef^{\not\sim_{\bn}}_{k}\|_{\BL^p(O')}^p.
    \end{equation}
\end{enumerate}
Recall \eqref{ultimate-related-fcn} that $f^{\not\sim_{\bn}}_{k}$ is a sum of wave packets $f_T$ for $T\in\ZT^{\not\sim_{\bn}}_{k}$, where 
\begin{equation}
    \ZT^{\not\sim_{\bn}}_{k}\subset\bigcap_{j=1}^{m_2}\ZT^{\not\sim_{\bn}}_{k}(j)\subset\ZT^{\not\sim_{\bn}}_{k}(t)
\end{equation}
for any $t$. Hence from \eqref{alpha-lambda-4}, \eqref{non-al-la-1}, \eqref{non-al-la-2}, and \eqref{non-al-la-3}, we know that at every scale $r_t$ there is a partition 
\begin{equation}
    \ZT^{\not\sim_{\bn}}_{k}=\bigsqcup_{\al,\la}\ZT^{\not\sim_{\bn}}_{k,\al,\la, t}
\end{equation}
with $\ZT^{\not\sim_{\bn}}_{k,\al,\la, t}\subset\ZT^{\not\sim_{\bn}}_{k,\al,\la}(t)$ (see \eqref{non-al-la-1}). We would like to consider the generators of the algebra created by the sets $\{\ZT^{\not\sim_{\bn}}_{k,\al,\la, t}\}_{\al,\la, t}$. To do so, define two vectors $\vec\al=(\al(1),\ldots,\al(m_2))$, $\vec\la=(\la(1),\ldots,\la(m_2))$, and the set
\begin{equation}
\label{generators}
    \ZT_{k,\vec{\al},\vec{\la}}^{\not\sim_{\bn}}=\bigcap_{t}\ZT^{\not\sim_{\bn}}_{k,\al(t),\la(t), t}.
\end{equation}
Then the sets $\{\ZT_{k,\vec{\al},\vec{\la}}^{\not\sim_{\bn}}\}_{\vec{\al},\vec{\la}}$ form a disjoint union of $\ZT^{\not\sim_{\bn}}_{k}$. Since the number of choices for $(\vec{\al},\vec{\la})$ is bounded above by $(\log R)^{O(n^2)}=(\log R)^{O(\e^{-10})}$, by pigeonholing there are a pair $(\vec{\al},\vec{\la})$ and a subset $\co_{leaf,k,\vec{\al},\vec{\la}}^{\not\sim}\subset\co_{leaf,k}^{\not\sim}$ with $|\co_{leaf,k,\vec{\al},\vec{\la}}^{\not\sim}|\gtrapprox|\co_{leaf,k}^{\not\sim}|$ so that the function $f^{\not\sim_{\bn}}_{k,\vec{\al},\vec{\la}}=\sum_{T\in\ZT^{\not\sim_{\bn}}_{k,\vec{\al},\vec{\la}}}f_T$ satiafies
\begin{equation}
    \sum_{O'\in\co_{leaf}}\|Ef\|_{\BL^p(O')}^p\lesssim R^{O(\de)}\sum_{O'\in\co_{leaf,k,\vec{\al},\vec{\la}}^{\not\sim}}\|Ef^{\not\sim_{\bn}}_{k,\vec{\al},\vec{\la}}\|_{\BL^p(O')}^p,
\end{equation}
and that $\|Ef^{\not\sim_{\bn}}_{k,\vec{\al},\vec{\la}}\|_{\BL^p(O')}$ are about the same for all $O'\in \co_{leaf,k,\vec{\al},\vec{\la}}^{\not\sim}(t)$ up to a constant multiple.

To ease notation, we set 
\begin{equation}
\label{g}
    g=f^{\not\sim_{\bn}}_{k,\vec{\al},\vec{\la}}, \hspace{.5cm}    \ZT_g[R]:=\ZT^{\not\sim_{\bn}}_{k,\vec{\al},\vec{\la}}, \hspace{.5cm}\co_{leaf}^{g}=\co_{leaf,k,\vec{\al},\vec{\la}}^{\not\sim}
\end{equation}
in the rest of the paper for simplicity. Hence $g=\sum_{T\in\ZT_{g}[R]}f_T$, $\|Eg\|_{\BL^p(O')}$ are about the same up to a constant multiple for all $O'\in\co_{leaf}^g$,
\begin{equation}
\label{g-estimate-1}
    \sum_{O'\in\co_{leaf}}\|Ef\|_{\BL^p(O')}^p\lesssim R^{O(\de)}\sum_{O'\in\co_{leaf}^{g}}\|Eg\|_{\BL^p(O')}^p,
\end{equation}
and also 
\begin{equation}
\label{leaf-g}
    |\co_{leaf}^{g}|\gtrapprox R^{-O(\e_0)}|\co_{leaf}|.
\end{equation}
We remark that the leaves in $\co_{leaf}^g$ are all contained in the $R^{1-\e_0}$-ball $B_{k}$.

\smallskip

We plug the new function $Eg$ back to the tree structure $\co_{tree}$ obtained in Lemma \ref{algebraic-lem-1}, and want to find the scale that the tangent case dominates. The following lemma is an analogy of Lemma 3.9 in \cite{Wang-restriction-R3}. Its proof is also similar, so the detail is omitted.

\begin{lemma}
\label{algebraic-lem-2}

Recall the tree structure obtained in Lemma \ref{algebraic-lem-1}. For the function $Eg$ in \eqref{g}, either of the following happens:
\begin{enumerate}
    \item There are $\geq 2R^{-\e_0}$-fraction of leaves $O'\in\co_{leaf}^g$ so that
    \begin{equation}
        \|Eg\|_{\BL^p(O')}\leq R^{\e_0}\|Eg_{O'}\|_{\BL^p(O')},
    \end{equation}
    which corresponds to the case that the polynomial partitioning iteration does not stop before reaching the smallest scale $R^\de$.
    \item We can choose the smallest integer $t$, $1\leq t\leq n$ (see Lemma \ref{algebraic-lem-1} for $n$), so that there is a subset $\co_{leaf}^g(t)\subset \co_{leaf}^g$ with
    \begin{equation}
    \label{leaf-g-t}
        |\co_{leaf}^g(t)|\geq 2R_{j_{t-1}}^{-\e_0}|\co_{leaf}^g|,
    \end{equation}
    and for each $O'\in\co_{leaf}^g(t)$, 
    \begin{equation}
        \|Eg\|_{\BL^p(O')}\leq R_{j_{t-1}}^{\e_0}\|Eg_{S_t}\|_{\BL^p(O')},
    \end{equation}
    while for all $1\leq l<t$, 
    \begin{equation}
    \label{small-step-l}
        \|Eg\|_{\BL^p(O')}\geq R_{j_{l-1}}^{\e_0}\|Eg_{S_l}\|_{\BL^p(O')}.
    \end{equation}
\end{enumerate}

If the second case holds, then there are a degree $D=d^{j_t}$, a collection of scale-$t$ fat surfaces, which we still denote by $\cs_t$, so that each $S_t\in\cs_t$ contains about the same amount of leaves in $\co^g_{leaf}(t)$, and (by \eqref{partitioning-blp-norm2} and \eqref{partitioning-blp-norm})
\begin{equation}
\label{cs-t}
    |\cs_t|\gtrapprox R^{-O(\de)} D^3,
\end{equation} 
and that the set $\{O\in \co^g_{leaf}(t):O\subset \bigcup_{S_t\in\cs_t}S_t\}$ contains a fraction $\gtrapprox1$ leaves in $\co^g_{leaf}(t)$. Still denote this fraction of leaves by $\co^g_{leaf}(t)$. In addition, from Lemma \ref{algebraic-lem-1} item (6) we have
\begin{equation}
\label{polynomial-l2}
    \sum_{S_t\in\cs_t}\|g_{S_t}\|_2^2\lesssim D\|g\|_2^2.
\end{equation}

\end{lemma}

If the first case in Lemma \ref{algebraic-lem-2} holds, then similar to Lemma 4.1 in \cite{Wang-restriction-R3}, one can prove \eqref{main-esti} for $p>3$. Therefore, let us assume that the second case in lemma \ref{algebraic-lem-2} is true, so we are given a $t$ and a corresponding scale $r_t$.

\begin{remark}
\label{remark-regularity-1}
\rm 

Since $\|Eg\|_{\BL^p(O')}$ are about the same (see the beginning of this section), without loss of generality let us assume $\|Eg_{S_t}\|_{\BL^p(O')}$ are also about the same up to a constant multiple for all $O'\in\co_{leaf}^g(t)$.
\end{remark}

Recall in \eqref{containing-chain} that each leaf $O'$ is within a containing chain $O'\subset S_n\subset S_{n-1}\subset\cdots\subset S_1\subset B_R$. Similar to \eqref{pointwise-cell-2} and \eqref{iterative-equation}, we have for $S_j\subset O_j\subset S_{j-1}$, 
\begin{equation}
    Eg_{O_j}(x)= Eg_{S_{j-1},trans}(x),\hspace{.5cm}x\in O_j,
\end{equation}
and hence for any subcollection $\ZT\subset\ZT_{O_j}$,
\begin{equation}
\label{iterative-equation-2}
    Eg_{O_j}^\ZT(x)= Eg^\ZT(x)-\Big(\sum_{l=1}^{j-1}Eg_{S_l}^\ZT(x)\Big), \hspace{.5cm}x\in O_j.
\end{equation}
For the $S_t$ we have
\begin{equation}
\label{cell-function}
    Eg_{S_t}= Eg^{\ZT_{S_t}}-\Big(\sum_{l=1}^{j-1}Eg_{S_l}^{\ZT_{S_t}}\Big).
\end{equation}
Note that \eqref{fat-fat} still holds when $f$ is replaced by $g$. Hence the contribution from  $Eg_{S_l}^{\ZT_{S_t}}$ is also dominated by the contribution from $Eg_{S_l}$, which is negligible since step $t$ is the first step that tangential case dominates (see \eqref{small-step-l}). Thus, for each leaf $O'\in\co_{leaf}^g(t)$, one has similar to \eqref{fat-fat-cor} that
\begin{equation}
    \|Eg_{S_t}\|_{\BL^p(O')}\sim\|Eg^{\ZT_{S_t}}\|_{\BL^p(O')}.
\end{equation}
This leads to, by an abuse on notation $\|Eg\|_{\BL^p(B_R)}^p=\sum_{O'\in\co_{leaf}^g(t)}\|Eg\|_{\BL^p(O')}^p$,
\begin{align}
\label{non-related-BLp}
    \|Eg\|_{\BL^p(B_R)}^p&\lesssim R^{O({\de})}\sum_{S_t}\sum_{\substack{O'\in\co_{leaf}^g(t),\\O'\subset S_t}}\|Eg_{S_t}\|_{\BL^p(O')}^p\\
    &\sim R^{O({\de})}\sum_{S_t}\sum_{\substack{O'\in\co_{leaf}^g(t),\\O'\subset S_t}}\|Eg^{\ZT_{S_t}}\|_{\BL^p(O')}^p.
\end{align}

The step $t$ corresponds to a scale $r_t=R_{j_t-1}$, which we denote by $r$ in the rest of the paper for brevity. If $r\leq R^{1/2}$ then by Lemma 7.5 in \cite{Wang-restriction-R3} we know that \eqref{main-esti} is true when $p>3.2=48/15$ (see also Remark \ref{broom-remark-0}). Hence, let us assume $r\geq R^{1/2}$ from now on. By Lemma \ref{algebraic-lem-1} item (4), we know that the fat surfaces in $\cs_t$ are essentially $r^{1/2+\de}$-separated. Let us put this into a lemma for later use.
\begin{lemma}
\label{disjointness-of-cells}
The sets $\{N_{r^{1/2+\de}}(S_t)\}_{S_t\in\cs_t}$ are finitely overlapping.
\end{lemma}

Note that by \eqref{generators}, \eqref{g}, the set $\ZT_g[R]$ is contained in $\ZT^{\not\sim_{\bn}}_{k,\al,\la}(t)$ for some $\al,\la$ (see \eqref{non-al-la-1}, \eqref{non-al-la-3} for $\ZT^{\not\sim_{\bn}}_{k,\al,\la}(t)$). The following two lemmas are the consequences of the two-ends reduction in Section \ref{two-ends-section}. We remark that from the sorting of leaves (Section \ref{sorting-section}) and from \eqref{leaf-g}, the assumptions $|\cq_{\la_3}^{k}|\geq R^{-O(\de^2)} |\cq_{\la_3}|$ and $|\cb^{k}(t)|\geq R^{-O(\de^2)} |\cb(t)|$ in Corollary \ref{corollary-kakeya-1} and Corollary \ref{corollary-kakeya-2} respectively hold readily.

\begin{lemma}
\label{first-method-lem}
Suppose $R^{1/2}\leq r=r_t\leq R^{2/3}$. Let $\rho_j$ be that $\rho_j\geq R/r^{1/2}\geq \rho_{j-1}$ (see \eqref{scales-1} for $\rho_j$). Then from Section \ref{second-related-tubes} (Corollary \ref{corollary-kakeya-2} in particular) we have
\begin{enumerate}
    \item Two dyadic numbers $\al,\la$.
    \item A collection of $R/r^{1/2}\times R/r^{1/2}\times R$-tubes $\wt\cT^{\not\sim_{\bn}}_{k,\al,\la}(t)$ so that each $R$-tube in $\ZT_{g}[R]$ (see \eqref{g} for $\ZT_g[R]$) is contained in some fat tube $\wt T\in\wt\cT^{\not\sim_{\bn}}_{k,\al,\la}(t)$.
    \item  A collection of  $R/r^{1/2}$-balls $\cb_\la^{k}(t)$ so that each ball $B\in\cb_\la^{k}(t)$ intersects $\lesssim R^{O(\de)}r^{1/4}\la^{-1}\ka_2(j)$ tubes $\wt T\in\wt\cT^{\not\sim_{\bn}}_{k,\al,\la}(t)$, and $   |\cb^{k}(t)\setminus\cb_\la^{k}(t)|\lessapprox R^{-\Om(\de)}|\cb^{k}(t)|$.
    \item Every leaf in $\co_{leaf}^{g}$ is contained in some of $B\in\cb^{k}(t)$.
    \item Every $T\in\ZT_{g}[R]$ intersects $\lesssim \la r^{1/2}$ balls in $\cb_\la^{k}(t)$.
\end{enumerate}
Note that each $B\in\cb^{k}(t)$ contains about the same amount of leaves in $\co_{leaf}$ (see Section \ref{second-sorting}). As a consequence of (3), (4), and  $|\co_{leaf}^g(t)|\gtrsim R^{-O(\e_0)}|\co_{leaf}^g|\gtrapprox R^{-O(\e_0)}|\co_{leaf}|$ from Lemma \ref{algebraic-lem-2} and \eqref{leaf-g} respectively, the set $\{O'\in\co_{leaf}^g:O'\subset B\text{ for some }B\in\cb_\la^{k}(t)\}$ contains a fraction $\gtrapprox1$ of leaves in $\co_{leaf}^{g}(t)$. Hence, if still denote this fraction of leaves by $\co_{leaf}^g(t)$, then we have
\begin{enumerate}
    \item[(6)]  Every leaf in $\co_{leaf}^{g}(t)$ is contained in some of $B\in\cb^{k}_\la(t)$.
\end{enumerate}
\end{lemma}

\begin{lemma}
\label{second-method-lem}
Suppose $R\geq r=r_t\geq R^{2/3}$. Let $\rho_j$ be that $\rho_j\geq R/r^{1/2}\geq \rho_{j-1}$ (see \eqref{scales-1} for $\rho_j$). We have from Section \ref{first-sorting}, Section \ref{first-related-tubes} (Corollar \ref{corollary-kakeya-1} in particular) and Lemma \ref{algebraic-lem-2} that
\begin{enumerate}
    \item Three dyadic numbers $\la_2,\la_3,\la_6$.
    \item A collection of scale-$r$ fat surfaces $\cs_t$ and a collection of $R^{1/2}$-balls $\cq_{\la_3}$ so that by Lemma \ref{lem: conservation of lambdas}, Lemma \ref{algebraic-lem-2} (\eqref{cs-t} in particular), every $S_t\in\cs_t$ is contained in $B_k$  (see above \eqref{section-4-eq-1} for $B_{k}$), and $\la_6|\cs_t|\approx R^{O(\de)} \la_3|\cq_{\la_3}|$.
    \item A set of $r^{1/2}$-cubes $\bq_t$. For every $Q\in\cq_{\la_3}, S_t\in\cs_t$, a set of $r^{1/2}$-cubes $\bq_{S_t,Q}=\{q\in\bq_t:q\subset S_t\cap Q\}$ with $|\bq_{S_t,Q}|\sim\la_2$ so that any leaf in $\co_{leaf}^g$ contained in $Q\cap S_t$ is contained by some $q\in\bq_{S_t,Q}$. Also, $|\cs_t(Q)|\sim \la_3$ (see \eqref{la-3-t}) and the set $\bq_{S_t}=\{q\in\bq_t:q\subset S_t\}$ has cardinality $\lesssim \la_2\la_6$.
    \item A set $\cq_{\la_3}^{k}=\{Q\in\cq_{\la_3}:Q\subset B_{k}\}$. Since $|\cq_{\la_3}^{k}|\geq R^{-O(\de^2)} |\cq_{\la_3}|$ (mentioned above Lemma \ref{first-method-lem}), we still have $\la_6|\cs_t|\approx R^{O(\de)} \la_3|\cq_{\la_3}^k|$.
    \item Two dyadic numbers $\al,\la$, a subset $\cq_{\la_3,\la}^{k}$ with $|\cq_{\la_3}^{k}\setminus\cq_{\la_3,\la}^{k}|\lessapprox R^{-\Om(\de)}|\cq_{\la_3}^k|$, so that each $Q\in\cq_{\la_3,\la}^{k}$ intersects $\lessapprox R^{O(\de)}R^{1/4}\la^{-1}\la_1$ tubes in $\ZT_g[R]$. Moreover, an $L^1$ estimate \eqref{l1-estimate}.
    \item Every leaf in $\co_{leaf}^g$ is contained in some $Q\in\cq_{\la_3}^{k}$.
\end{enumerate}
Note that each $Q\in\cq_{\la_3}^{k}$ contains about the same amount of leaves in $\co_{leaf}$ (see Section \ref{first-sorting}). As a consequence of (5), (6), and $|\co_{leaf}^g(t)|\gtrsim R^{-O(\e_0)}|\co_{leaf}^g|\gtrapprox R^{-O(\e_0)}|\co_{leaf}|$ from Lemma \ref{algebraic-lem-2} and \eqref{leaf-g} respectively, the set $\{O'\in\co_{leaf}^g:O'\subset Q\text{ for some }Q\in\cq_{\la_3,\la}^{k}\}$ contains a fraction $\gtrapprox R^{-O(\e_0)}$ of leaves in $\co_{leaf}^{g}(t)$. Hence, if still denote this fraction of leaves by $\co_{leaf}^g(t)$, then we have
\begin{enumerate}
    \item[(7)] Every leaf in $\co_{leaf}^{g}(t)$ is contained in some of $Q\in\cq_{\la_3,\la}^{k}$.
\end{enumerate}
In addition, since again each $Q\in\cq_{\la_3}^k$ contains about the same amount of leaves in $\co_{leaf}$, by pigeonholing we can find a subset $\bar\cq_{\la_3,\la}^{k}\subset \cq_{\la_3,\la}^{k}$ so that 
\begin{enumerate}
    \item[(8)] $|\bar\cq_{\la_3,\la}^{k}|\gtrapprox R^{-O(\de)}|\cq_{\la_3,\la}^{k}|$, and the set $\{O'\in\co_{leaf}^g(t):O'\subset Q\text{ for some }Q\in\bar\cq_{\la_3,\la}^{k}\}$ contains a fraction $\gtrapprox R^{-O(\e_0)}$ of leaves in $\co_{leaf}^{g}(t)$, and each $Q\in\bar\cq_{\la_3,\la}^{k}$ intersects $\sim\mu\lessapprox R^{O(\de)}R^{1/4}\la^{-1}\la_1$ tubes in $\ZT_g[R]$ (see item (5)). Moreover, by \eqref{l1-estimate} we have $\mu|\cq_{\la_3,\la}^{k}|\lessapprox R^{O(\e_0)}\mu|\bar\cq_{\la_3,\la}^{k}|\lessapprox R^{O(\e_0)} \la\la_1R^{3/2}$.
\end{enumerate}

\end{lemma}

\section{First method}
\label{first-method}

The first method will handle the case $R^{1/2}\leq r\leq R^{2/3}$. Recall Lemma \ref{algebraic-lem-2} (\eqref{cs-t} in particular) that each $S_t\in\cs_t$ contains about the same amount of leaves in $\co_{leaf}^g(t)$. We have from \eqref{non-related-BLp}
\begin{align}
\label{non-related-BLp-1}
    \|Eg\|_{\BL^p(B_R)}^p&\lesssim R^{O(\de)}\sum_{S_t\in\cs_t}\sum_{\substack{O'\in\co_{leaf}^{g}(t),\\O'\subset S_t}}\|Eg_{S_t}\|_{\BL^p(O')}^p\\
    &\sim R^{O(\de)}\sum_{S_t\in\cs_t}\sum_{\substack{O'\in\co_{leaf}^{g}(t),\\O'\subset S_t}}\|Eg^{\ZT_{S_t}}\|_{\BL^p(O')}^p.
\end{align}

\subsection{Some more preparations} 

Cover each cell $S_t\in\cs_t$ with non-overlapping $r^{1/2}$-cubes $q$, and let $\bq_{S_t}$ be the collection of all these cubes. Similar to \eqref{bq-S-t},
\begin{equation}
\label{bq-S-t-2}
    |\bq_{S_t}|\lessapprox r.
\end{equation}
Note that by Lemma \ref{disjointness-of-cells}, the collections $\{\bq_{S_t}\}_{S_t}$ are finitely overlapped. Write
\begin{align}
\label{first-method-beginning}
    \|Eg\|_{\BL^p(B_R)}^p&\lesssim R^{O(\de)}\sum_{S_t\in\cs_t}\sum_{q\in\bq_{S_t}}\sum_{\substack{O'\in\co_{leaf}^{g}(t),\\O'\subset q}}\|Eg_{S_t}\|_{\BL^p(O')}^p\\
    &\sim R^{O(\de)}\sum_{S_t\in\cs_t}\sum_{q\in\bq_{S_t}}\sum_{\substack{O'\in\co_{leaf}^{g}(t),\\O'\subset q}}\|Eg^{\ZT_{S_t}}\|_{\BL^p(O')}^p.
\end{align}
By dyadic pigeonholing, we can find a dyadic number $\eta_1$, a subset $\bq_{S_t}^1\subset \bq_{S_t}$, a subset $\cs_t^1\subset\cs_t$, and a subset $\co_{leaf}^{g(1)}\subset\co_{leaf}^{g}(t)$ so that
\begin{enumerate}
    \item For all $S_t\in\cs_t^1$
    \begin{equation} 
    \label{bq-S-t-esti}
        |\bq_{S_t}^1|\sim\eta_1.
    \end{equation}
    \item Each $O'\in\co_{leaf}^{g(1)}$ is contained by some $r^{1/2}$-cube in some $\bq_{S_t}^1$, and each $q\in\bq_{S_t}^1$ contains about the same amount of leaves in $\co_{leaf}^{g(1)}$ up to a constant multiple for $S_{t}\in\cs_t^1$.
    \item It holds that
    \begin{equation}
    \label{leaf-g-2}
        |\co_{leaf}^{g(1)}|\gtrapprox|\co_{leaf}^{g}|.
    \end{equation}
\end{enumerate}
Hence we also have
\begin{equation}
\label{cs-t-2}
    |\cs_t^1|\gtrapprox|\cs_t|.
\end{equation}

\smallskip 

Recall Lemma \ref{first-method-lem}. Since each $R$-tube $T\in\ZT_g[R]$ is contained by some fat tube $\wt T\in\wt\cT^{\not\sim_{\bn}}_{k,\al,\la}(t)$, we know that each $r$-tube $T'\in\ZT_{S_t}$ with $g_{T'}\not=0$ is associated to one (and at least one) fat tube $\wt T\in\wt\cT^{\not\sim_{\bn}}_{k,\al,\la}(t)$ (see also Lemma 7.1 in \cite{Guth-II}). Thus, after deleting those $T'\in\ZT_{S_t}$ with $g_{T'}=0$, for an $S_t$ and for some $R/r^{1/2}$-ball $B\in\cb_\la^{k}(t)$ that $S_t\subset B$, the number of scale-$r$ directional caps in $\ZT_{S_t}$ is bounded above by the number of tubes in $\wt\cT^{\not\sim_{\bn}}_{k,\al,\la}(t)$ intersecting $B$, which is $\lessapprox R^{O(\de)}r^{1/4}\la^{-1}\ka_2(j)$ by Lemma \ref{first-method-lem}.

By the polynomial Wolff axiom, there is another upper bound for the number of scale-$r$ directional caps in $\ZT_{S_t}$, which is $R^{O(\de)}r^{1/2}$. This follows from Lemma 4.9  in \cite{Guth-R3} (see also Lemma \ref{polynomial-plane-lemma}). Therefore, compared to (7.26) in \cite{Wang-restriction-R3}, we obtain the following refinement: for every $S_t\subset B$ where $B\in\cb_\la^{k}(t)$, one has
\begin{equation}
    \|Eg^{\ZT_{S_t}}\|_{L^2(B_{S_t})}^2\lessapprox R^{O(\de)}\min\{r^{1/4}\la^{-1}\ka_2(j),r^{1/2}\}\sup_{\theta'\in\Theta[r]}\|Eg^{\ZT_{{S_t},\theta'}}\|_{L^2(B_{S_t})}^2,
\end{equation}
where $B_{S_t}$ is the $r$-ball containing ${S_t}$. Combining it with the broom estimate \eqref{broom-esti-2} (see also Remark \ref{broom-remark-0}) and the fact that $\ka_2(j)\geq1$,
\begin{align}
\label{refined-broom}
    \|Eg^{\ZT_{S_t}}\|_{L^2(B_{S_t})}^2&\lessapprox R^{O(\e_0)}\Big(\frac{r}{R}\Big)^{3/2}\min\{r^{1/4}\la^{-1},r^{1/2}\}\sup_{\theta'\in\Theta[r]}\|Ef_{\theta'}\|_2^2\\ \nonumber
    &\lesssim R^{O(\e_0)}\Big(\frac{r}{R}\Big)^{1/2}\min\{r^{1/4}\la^{-1},r^{1/2}\}\sup_{\theta}\|f_\theta\|_{L^2_{ave}}^2.
\end{align}
Estimate \eqref{refined-broom} works well when $\la$ is big. While when $\la$ is small, we need another argument.

\medskip

Now for a fixed $S_t\in\cs_t^1$, let us focus on 
\begin{equation}
    \|Eg_{S_t}\|_{\BL^p(S_t')}^p\sim\|Eg^{\ZT_{S_t}}\|_{\BL^p(S_t')}^p,
\end{equation}
where (see \eqref{bq-S-t-esti} for $\bq_{S_t}^1$)
\begin{equation}
\label{st'}
    S_t'=\bigcup_{q\in\bq_{S_t}^1}q.
\end{equation}
What follows is another dyadic pigeonholing. Note that the function $Eg^{\ZT_{S_t}}$ is a sum of wave packets at scale $r$ inside the $r$-ball containing the cell ${S_t}$. A heuristic understanding for  $Eg^{\ZT_{S_t}}$ is that if denoting $\ZT_{S_t}=\bigcup_{\theta'\in\Theta[r]}\ZT_{{S_t},\theta'}[r]$, then inside ${S_t}$ 
\begin{equation}
    Eg^{\ZT_{S_t}}\approx\sum_{\theta'\in\Theta[r]}\sum_{T'\in\ZT_{{S_t},\theta'}[r]} Eg_{\theta'}\Id_{T'}.
\end{equation}
For each $T'\in\ZT_{S_t}$ define
\begin{equation}
    T'(\bq_{{S_t}}^1)=|\{q\in\bq_{{S_t}}^1:q\cap T'\not=\varnothing\}|.
\end{equation}
We would like to use pigeonholing to find a fraction $\approx1$ of fat surfaces $S_t\in\cs_t^1$, so that for all tubes $T'\in\ZT_{S_t}$ the quantities $T'(\bq_{{S_t}}^1)$ are about the same up to a constant multiple. To do so, for each dyadic number $\eta_2\in[1,Cr^{1/2}]$, we partition the collection $\ZT_{S_t}$ as
\begin{equation}
    \ZT_{S_t}=\bigsqcup_{\eta_2}\ZT_{{S_t},\eta_2},
\end{equation}
where $\ZT_{{S_t},\eta_2}$ is a subcollection defined as
\begin{equation}
\label{eta-2}
    \ZT_{{S_t},\eta_2}=\{T'\in\ZT_{S_t}:T'(\bq_{S_t}^1)\sim\eta_2\}.
\end{equation}
Using this partition, we can write
\begin{equation}
    Eg^{\ZT_{S_t}}=\sum_{\eta_2}Eg^{\ZT_{S_t},\eta_2}
\end{equation}
where $g^{\ZT_{S_t},\eta_2}$ is the sum of wave packets from $\ZT_{S_t,\eta_2}$. Hence
\begin{equation}
\label{eta-2-function}
    \|Eg^{\ZT_{S_t}}\|_{\BL^p(q)}^p\lessapprox \sum_{\eta_2}\|Eg^{\ZT_{S_t},\eta_2}\|_{\BL^p(q)}^p.
\end{equation}
By pigeonholing on $\eta_2$ and dyadic pigeonholing on the value $\|Eg^{\ZT_O,\eta_2}\|_{\BL^p(q)}^p$, we can choose a uniform $\eta_2$ such that for a fraction $\gtrapprox1$ of fat surfaces $S_t\in\cs_t^1$ and a fraction $\gtrapprox1$ of $r^{1/2}$-balls $q\in \bq_{S_t}^1$, 
\begin{enumerate}
    \item $\|Eg^{\ZT_{S_t}}\|_{\BL^p(q)}^p\lessapprox \|Eg^{\ZT_{S_t},\eta_2}\|_{\BL^p(q)}^p$.
    \item For all $q$ appearing in the first item, $\|Eg^{\ZT_O,\eta_2}\|_{\BL^p(q)}^p$ are about the same up to a constant multiple.
\end{enumerate}
Recall that we have made some uniform assumptions on $|\bq_{S_t}^1|$ and the number of leaves in each $r^{1/2}$-cube in $\bq_{S_t}^1$. (see \eqref{bq-S-t-esti} and the statement below it). To avoid extra notations, let us assume the above is true for all ${S_t}\in\cs_t^1$ and all $q\in\bq_{{S_t}}^1$, without loss of generality.

Now by Lemma \ref{polynomial-plane-lemma} and Lemma \ref{planar-refinement-lem}, and the fact that broad-norm is essentially dominated by bilinear norm, one has (see \eqref{bq-S-t-esti}, \eqref{eta-2} for $\eta_1,\eta_2$ and \eqref{st'} for $S_t'$)
\begin{align}
\label{variety-esti-1-final}
    \|Eg^{\ZT_{S_t}}\|_{\BL^p({S_t'})}^p&\lesssim\sum_{q\in \bq_{S_t}^1}\|Eg^{\ZT_{S_t},\eta_2}\|_{\BL^p(q)}^p\\ \nonumber
    &\lessapprox R^{O(\de)}r^{3/2-p}\eta_1^{1-p/2}\min\{\eta_1^{p/4},\eta_2^{p/2}\}\|Eg^{\ZT_{S_t},\eta_2}\|_{L^2(B_{S_t})}^p.
\end{align}
Here $B_{S_t}$ is the $r$-ball containing ${S_t}$. Summing up all cells $S_t\in\cs_t^1$ we get
\begin{align}
\label{reduction-1}
    \sum_{S_t\in\cs_t^1}\|Eg^{\ZT_{S_t},\eta_2}\|_{\BL^p({S_t'})}^p\lessapprox R^{O(\de)}r^{3/2-p}\eta_1^{1-p/2}\!\min\{\eta_1^{p/4},\eta_2^{p/2}\}\!\!\!\sum_{{S_t}\in\cs_t^1}\!\!\!\|Eg^{\ZT_{S_t},\eta_2}\|_{L^2(B_{S_t})}^p.
\end{align}

\smallskip

We also would like to give an upper bound for $\sum_{S_t\in\cs_t^1}\|Eg^{\ZT_{S_t},\eta_2}\|_{L^2(B_{S_t})}^2$. From Lemma \ref{first-method-lem} we know that each ${S_t}\in\cs_t^1$ is contained in some $B\in\cb_\la^{k}(t)$. Write
\begin{equation}
    \sum_{S_t\in\cs_t^1}\|Eg^{\ZT_{S_t},\eta_2}\|_{L^2(B_{S_t})}^2\lesssim\sum_{B\in\cb_\la^{k}(t)}\sum_{\substack{{S_t}\in\cs_t^1,\\ {S_t}\subset B}}\|Eg^{\ZT_{S_t},\eta_2}\|_{L^2(B_{S_t})}^2.
\end{equation}
Notice that by our definition of $\ZT_{{S_t},\eta_2}$ in \eqref{eta-2}, any $r$-tube belongs to at most $r^{1/2}\eta_2^{-1}$   distinct collections $\{\ZT_{{S_t},\eta_2}\}_{S_t}$. By $L^2$-orthogonality, for each $Rr^{-1/2}$-ball $B\in\cb_\la^{k}(t)$ we have 
\begin{align}
    \sum_{\substack{S_t\in\cs_t^1,\\ S_t\subset B}}\|Eg^{\ZT_{S_t},\eta_2}\|_{L^2(B_{S_t})}^2&\lesssim r^{1/2}\eta_2^{-1}\|Eg^{\ZT_{S_t},\eta_2}\|_{L^2(B)}^2\\
    &\lesssim r^{1/2}\eta_2^{-1}\sum_{T\in\ZT_g[R]} \|Ef_T\|_{L^2(B)}^2.
\end{align}
Since each $T\in\ZT_g[R]$ intersects at most $\la r^{1/2}$ balls in $\cb_\la^{k}(t)$ (see Lemma \ref{first-method-lem}), we can sum up all $B\in\cb_\la^{k}(t)$ and get
\begin{equation}
\label{l2-shading}
    \sum_{{S_t}\in\cs_t^1}\|Eg^{\ZT_{S_t},\eta_2}\|_{L^2(B_{S_t})}^2\lesssim r^{1/2}\eta_2^{-1}\la\|Ef\|_2^2.
\end{equation}
This suggests a gain when $\la$ is small.

\subsection{Wrap up}

Finally, let us wrap up all the information we get so far. We are going to estimate $\|Eg\|_{\BL^p(B_R)}$ in three ways. from \eqref{cs-t} and \eqref{cs-t-2} we know that $|\cs_t^1|\gtrapprox R^{-O(\de)}|\cs_t|\gtrapprox R^{-O(\de)}D^3$.

\subsubsection{Case 1.}

Recall \eqref{reduction-1}, \eqref{l2-shading} and \eqref{refined-broom} (recall also \eqref{st'}). Since $\ZT_{{S_t},\eta_2}\subset\ZT_{S_t}$ (see \eqref{eta-2} for $\ZT_{{S_t},\eta_2}$), we have $\|Eg^{\ZT_{S_t},\eta_2}\|_{{L^2(B_{S_t})}}^2\lesssim\|Eg^{\ZT_{S_t}}\|_{{L^2(B_{S_t})}}^2$. Hence
\begin{align}
\nonumber
    &\sum_{{S_t}\in\cs_t^1}\|Eg^{\ZT_{S_t},\eta_2}\|_{\BL^p({S_t'})}^p\lessapprox R^{O(\de)}r^{3/2-p}\eta_1^{1-p/2}\min\{\eta_1^{p/4},\eta_2^{p/2}\}\!\!\!\sum_{{S_t}\in\cs_t^1}\!\!\!\|Eg^{\ZT_{S_t},\eta_2}\|_{L^2(B_{S_t})}^p\\ \nonumber
    &\lessapprox R^{O(\de)}r^{3/2-p}\eta_1^{1-p/2}\min\{\eta_1^{p/4},\eta_2^{p/2}\}\!\!\sum_{{S_t}\in\cs_t^1}\!\!\|Eg^{\ZT_{S_t},\eta_2}\|_{{L^2(B_{S_t})}}^2\sup_{{S_t}}\|Eg^{\ZT_{S_t},\eta_2}\|_{{L^2(B_{S_t})}}^{p-2}\\ \nonumber
    &\lesssim R^{O(\de)}C_1(R,r,\vec{\eta})\|f\|_2^{2}\sup_{\theta}\|f_{\theta}\|_{L^2_{ave}}^{p-2},
\end{align}
where, by \eqref{refined-broom} and \eqref{l2-shading} and Plancherel,
\begin{align}
    C_1(R,r,\vec{\eta})=&\, r^{3/2-p}\eta_1^{1-p/2}\min\{\eta_1^{p/4},\eta_2^{p/2}\}\cdot r^{1/2}\eta_2^{-1} \la R\\ \nonumber
    &\cdot \Big(\Big(\frac{r}{R}\Big)^{1/2}\min\{r^{1/4}\la^{-1},r^{1/2}\}\Big)^\frac{p-2}{2}.
\end{align}
Noticing $\min\{\eta_1^{p/4},\eta_2^{p/2}\}\leq \eta_1^{\frac{p-2}{4}}\eta_2$, we can simplify and reorder the terms in $C_1(R,r,\vec{\eta})$ to get
\begin{equation}
\label{C1}
    C_1(R,r,\vec{\eta})=Rr^{2-p}\la\eta_1^{\frac{2-p}{4}}\cdot\min\{r^{\frac{3}{4}}\la^{-1}R^{-\frac{1}{2}},rR^{-\frac{1}{2}}\}^\frac{p-2}{2}.
\end{equation}
It follows from  \eqref{first-method-beginning} and \eqref{eta-2-function} that
\begin{equation}
    \|Eg\|_{\BL^p(B_R)}^p\lesssim R^{O(\de)}C_1(R,r,\vec{\eta})\|f\|_2^{2}\sup_{\theta}\|f_{\theta}\|_{L^2_{ave}}^{p-2}.
\end{equation}

\subsubsection{Case 2.}

Recall the definition of $Eg_{S_t}$ in \eqref{cell-function} (see also \eqref{non-related-BLp}). By H\"older's inequality (a similar argument is given in \eqref{planar-refinement-1}) 
\begin{equation}
\label{old-lp-l2}
    \|Eg_{S_t}\|_{\BL^p(S_t')}^p\lessapprox R^{O(\de)}r^{3/2-p}\eta_1^{1-p/4}\|Eg_{S_t}\|_{L^2(B_{S_t})}^p,
\end{equation}
which yields (see \eqref{st'} for $S_t'$), by the fact that $\|Eg_{S_t}\|_2^2\lesssim r\|g_{S_t}\|_2^2$,
\begin{equation}
\label{reduction-2}
    \|Eg_{S_t}\|_{\BL^p({S_t'})}^2\lessapprox R^{O(\de)}r^{3/p-1}\eta_1^{2/p-1/2}\|g_{S_t}\|_2^2.
\end{equation}

The second comes from combining \eqref{reduction-1} and \eqref{reduction-2}. Note that
\begin{equation}
\label{cell-functions-split}
    \|Eg_{S_t}\|_{\BL^p(S_t')}^p\lesssim\|Eg_{S_t}\|_{\BL^p(S_t')}^2\|Eg^{\ZT_{S_t}}\|_{\BL^p(S_t')}^{p-2}.
\end{equation}
By \eqref{variety-esti-1-final} (with $\min\{\eta_1^{p/4},\eta_2^{p/2}\}\leq\eta_1^{p/4}$) and \eqref{reduction-2} one has
\begin{equation}
    \sum_{{S_t}\in\cs_t^1}\|Eg_{S_t}\|_{\BL^p(S_t')}^p\lessapprox R^{O(\de)} r^{\frac{5-2p}{2}}\eta_1^{\frac{4-p}{4}}\sum_{{S_t}\in\cs_t^1}\|g_{S_t}\|_{2}^2\|Eg^{\ZT_{S_t}}\|_{L^2(B_{S_t})}^{p-2}.
\end{equation}
From \eqref{polynomial-l2}, \eqref{non-related-BLp-1}, and \eqref{refined-broom} we thus have
\begin{align}
    \|Eg\|_{\BL^p(B_R)}^p\lessapprox  R^{O(\de)}C_2(R,r,\vec{\eta})\|f\|_2^{2}\sup_{\theta'}\|f_{\theta'}\|_{L^2_{ave}}^{p-2},
\end{align}
where
\begin{align}
\label{C2}
    C_2(R,r,\vec{\eta})=&Dr^{\frac{5-2p}{2}}\eta_1^{\frac{4-p}{4}}\cdot \min\{r^{\frac{3}{4}}\la^{-1}R^{-\frac{1}{2}},rR^{-\frac{1}{2}}\}^\frac{p-2}{2}.
\end{align}

\subsubsection{Case 3.}

One one hand, from \eqref{polynomial-l2}, \eqref{old-lp-l2}, and pigeonholing on $S_t$ we have (one could compare it with (7.38) in \cite{Wang-restriction-R3})
\begin{equation}
\label{case3-1}
    \|Eg\|_{\BL^p(B_R)}^p\lesssim R^{O(\de)}\sum_{S_t}\|Eg_{S_t}\|_{\BL^p({S_t'})}^p\lessapprox R^{O(\de)}D^{3-p}r^{\frac{3-p}{2}}\eta_1^{\frac{4-p}{4}}\|f\|_2^{p}.
\end{equation}

On the other hand, note that \eqref{old-lp-l2} is also true if $Eg_{S_t}$ is replaced by $Eg^{\ZT_{S_t}}$. Thus, from \eqref{polynomial-l2}, \eqref{old-lp-l2}, \eqref{cell-functions-split}, and the broom estimate \eqref{broom-esti-1} we have (one could compare it with (7.32) in \cite{Wang-restriction-R3})
\begin{align}
\label{case3-2}
    \|Eg\|_{\BL^p(B_R)}^p &\,\lesssim R^{O(\de)}\sum_{S_t}\|Eg_{S_t}\|_{\BL^p({S_t'})}^2\sup_{S_t}\|Eg^{\ZT_{S_t}}\|_{\BL^p({S_t'})}^{p-2}\\ \nonumber
    &\,\lesssim R^{O(\de)} DR^{\frac{2-p}{4}}r^{\frac{3-p}{2}}\eta_1^{\frac{4-p}{4}}\|f\|_2^{2}\sup_{\theta}\|f_{\theta}\|_{L^2_{ave}}^{p-2}.
\end{align}
Combining the above two estimates we finally get
\begin{equation}
    \|Eg\|_{\BL^p(B_R)}^p\lesssim R^{O(\de)} C_3(R,r,\vec{\eta})\|f\|_2^{2}\sup_{\theta}\|f_{\theta}\|_{L^2_{ave}}^{p-2},
\end{equation}
where
\begin{equation}
\label{C3}
    C_3(R,r,\vec{\eta})=\min\{D^{3-p},DR^{\frac{2-p}{4}}\}r^{\frac{3-p}{2}}\eta_1^{\frac{4-p}{4}}.
\end{equation}

\subsection{Numerology}

Recall \eqref{C1}, \eqref{C2} and \eqref{C3}. One wants to find the smallest $p$ so that the minimum of the following system is bounded above by $O(1)$
\begin{equation}
\label{system-4}
    \begin{cases}
    Rr^{2-p}\la\eta_1^{\frac{2-p}{4}}\cdot\min\{r^{\frac{3}{4}}\la^{-1}R^{-\frac{1}{2}},rR^{-\frac{1}{2}}\}^\frac{p-2}{2}\\[2ex]
    Dr^{\frac{5-2p}{2}}\eta_1^{\frac{4-p}{4}}\cdot \min\{r^{\frac{3}{4}}\la^{-1}R^{-\frac{1}{2}},rR^{-\frac{1}{2}}\}^\frac{p-2}{2}\\[2ex]
    \min\{D^{3-p},DR^{\frac{2-p}{4}}\}r^{\frac{3-p}{2}}\eta_1^{\frac{4-p}{4}}
    \end{cases}
\end{equation}
Simplify \eqref{system-4} by multiplying the first two terms and by using
\begin{equation}
    \min\{r^{\frac{3}{4}}\la^{-1}R^{-\frac{1}{2}},rR^{-\frac{1}{2}}\}^{p-2}\leq (r^{\frac{3}{4}}\la^{-1}R^{-\frac{1}{2}})(rR^{-\frac{1}{2}})^{p-3}
\end{equation}
so that we could get rid of the factor $\la$ and have
\begin{equation}
\label{system-1-1}
    \begin{cases}
    I=DR^\frac{4-p}{2}r^\frac{9-4p}{4}\eta_1^\frac{3-p}{2}\\[2ex]
    II=\min\{D^{3-p},DR^{\frac{2-p}{4}}\}r^{\frac{3-p}{2}}\eta_1^{\frac{4-p}{4}}
    \end{cases}
\end{equation}
Now by using $\min\{D^{3-p},DR^{\frac{2-p}{4}}\}\leq (D^{3-p})^t(DR^{\frac{2-p}{4}})^{1-t}$ with $t=\frac{3p-5}{(6p-15)(p-2)}$ and $\eta_1\lesssim r$ (see \eqref{bq-S-t-2} and \eqref{bq-S-t-esti}), one gets
\begin{equation}
    I\times II^\frac{6p-15}{10-3p}\lesssim R^\frac{14p-45}{12p-40},
\end{equation}
which is bounded above by 1 when $10/3>p>45/14$.

In conclusion, we have when $\log r/\log R\leq 2/3$ and $p>45/14$,
\begin{equation}
\label{result-1}
    \|Eg\|_{\BL^p(B_R)}^p\lesssim R^{O(\de)}\|f\|_2^{2}\sup_{\theta}\|f_{\theta}\|_{L^2_{ave}}^{p-2}.
\end{equation}

\section{Second method}
\label{second-method}

The second method will handle the case $r\geq R^{2/3}$ ($r=r_t$). Roughly speaking, we will use a square function estimate if each ``admissible" $R^{1/2}$-ball is ``associated to" a lot of fat surfaces $S_t$. Otherwise, we use the refined decoupling theorem.

Recall Section \ref{first-sorting} and Lemma \ref{second-method-lem}. Since $|\co_{leaf}^g(t)|\gtrapprox R^{-O(\e_0)}|\co_{leaf}|$ (see \eqref{leaf-g} and \eqref{leaf-g-t}), let us assume without loss of generality that each $q\in\bq_{S_t}$ contains about the same amount of leaves in $\cq_{leaf}^g(t)$ up to a constant multiple (see also Lemma \ref{lem: conservation of lambdas}). As a result, $\sum_{O\in\co_{leaf}^g,\,O'\subset q,\,q\in\bq_{S_t,Q}}\|Eg\|_{\BL^p(O')}^p$ are about the same up to a constant multiple for all $S_t\in\cs_t, Q\in\cq_{\la_3,\la}^{k}$. Note that by Lemma~\ref{lem: conservation of lambdas}, after further refinement, the parameters $\lambda_j$ change by at most a factor of $R^{4\delta}$.

Go back to \eqref{non-related-BLp}. Write
\begin{align}
\label{second-method-beginning}
    \|Eg\|_{\BL^p(B_R)}^p&\lesssim R^{O(\de)}\sum_{S_t\in\cs_t}\sum_{q\in\bq_{S_t}}\sum_{\substack{O'\in\co_{leaf}^{g}(t),\\O'\subset q}}\|Eg_{S_t}\|_{\BL^p(O')}^p\\ \nonumber
    &\sim R^{O(\de)}\sum_{S_t\in\cs_t}\sum_{q\in\bq_{S_t}}\sum_{\substack{O'\in\co_{leaf}^{g}(t),\\O'\subset q}}\|Eg^{\ZT_{S_t}}\|_{\BL^p(O')}^p.
\end{align}

\subsection{Case one} 
By Lemma \ref{second-method-lem} (item (7) in particular), we have from \eqref{second-method-beginning} that
\begin{equation}
\label{case-1}
    \|Eg\|_{\BL^p(B_R)}^p\lessapprox R^{O(\de)}\sum_{Q\in\cq_{\la_3,\la}^{k}}\sum_{S_t\in\cs_t(Q)}\sum_{q\in\bq_{S_t,Q}}\|Eg^{\ZT_{S_t}}\|_{\BL^p(q)}^p.
\end{equation}
For a fixed $Q\in\cq_{\la_3,\la}^{k}$ and a fixed $S_t\in\cs_t(Q)$, let us focus on 
\begin{equation}
    \sum_{q\in\bq_{S_t,Q}}\|Eg^{\ZT_{S_t}}\|_{\BL^p(q)}^p.
\end{equation}

What follows is another dyadic pigeonholing similar to the one near \eqref{eta-2}. For each dyadic number $\la_5\in[1,C(R/r)^{1/2}]$, we partition the collection $\ZT_{S_t}$ as
\begin{equation}
    \ZT_{S_t}=\bigsqcup_{\la_5}\ZT_{{S_t},\la_5},
\end{equation}
where $\ZT_{{S_t},\la_5}$ is a subcollection defined as (see Lemma \ref{second-method-lem} for $\bq_{S_t,Q}$) 
\begin{equation}
\label{la-5}
    \ZT_{{S_t},\la_5}=\ZT_{{S_t},\la_5}(Q)=\{T'\in\ZT_{S_t}:T'(\bq_{S_t,Q})\sim\la_5\}
\end{equation}
and $T'(\bq_{S_t,Q})$ is defined as
\begin{equation}
    T'(\bq_{S_t,Q})=|\{q\in\bq_{S_t,Q}:q\cap (T'\cap Q)\not=\varnothing\}|.
\end{equation}
Using this partition, inside $Q$ we can write
\begin{equation}
    Eg^{\ZT_{S_t}}=\sum_{\la_5}Eg^{\ZT_{{S_t},\la_5}}.
\end{equation}
Since there are at most $O(\log R)$     choices of $\la_5$, one gets
\begin{equation}
    \|Eg^{\ZT_{S_t}}\|_{\BL^p(q)}^p\lessapprox    \sum_{\la_5}\|Eg^{\ZT_{{S_t},\la_5}}\|_{\BL^p(q)}^p.
\end{equation}
By dyadic pigeonholing, we can choose a uniform $\la_5$ such that for a fraction $\gtrapprox1$ of $S_t\in\cs_t$, $Q\in\cq_{\la_3,\la}^{k}$, and $q\in \bq_{S_t,Q}$, $\|Eg^{\ZT_{S_t}}\|_{\BL^p(q)}^p\lessapprox    \|Eg^{\ZT_{S_t,\la_5}}\|_{\BL^p(q)}^p$.

By pigeonholing again, for a $\gtrapprox 1$ fraction of the remaining $\mathcal{S}_t$ and $\mathcal{Q}^k_{\lambda_3, \lambda}$, the quantity $\|Eg^{\ZT_{S_t,\la_5}}\|_{L^2(Q)}$ are about the same. We restrict our attention on those $S_t$ and $Q$. Recall that by Lemma~\ref{lem: conservation of lambdas}, after further refinement (see also the uniform assumption above \eqref{second-method-beginning}), the parameters $\lambda_j$ change by at most a factor of $R^{4\delta}$.

Now by Lemma \ref{polynomial-plane-lemma} and Lemma \ref{planar-refinement-lem} (in fact we only use \eqref{planar-refinement-2}), and the fact that broad-norm is essentially dominated by bilinear norm, one has
\begin{equation}
\label{variety-esti-2-final}
    \sum_{q\in \bq_{S_t,Q}}\|Eg^{\ZT_{S_t}}\|_{\BL^p(q)}^p\lessapprox R^{O(\de)} r^{3/2-p/2}R^{-p/4}\la_2^{1-p/2}\la_5^{p/2}\|Eg^{\ZT_{{S_t},\la_5}}\|_{L^2(Q)}^p.
\end{equation}
Here $\la_2$ is given in Lemma \ref{second-method-lem} ($|\bq_{S_t,Q}|\sim \la_2$), and $\la_5$ is defined in \eqref{la-5}.  Summing up all $Q\in\cq_{\la_3,\la}^{k}$ and all $O\in\cs_t(Q)$, we get via \eqref{case-1} that
\begin{align}
    \|Eg\|_{\BL^p(B_R)}^p\lessapprox R^{O(\de)} & r^{3/2-p/2}R^{-p/4}\la_2^{1-p/2}\la_5^{p/2}\\ \nonumber
    & \cdot\sum_{Q\in\cq_{\la_3,\la}^{k}}\sum_{S_t\in\cs_t(Q)}\|Eg^{\ZT_{S_t,\la_5}}\|_{L^2(Q)}^p.
\end{align}
Since $\|Eg^{\ZT_{{S_t},\la_5}}\|_{L^2(Q)}$ are about the same for all $S_t,Q$ and since $|\cs_t(Q)|\sim \la_3$ (see \eqref{la-3-t} and Lemma \ref{second-method-lem}),  we further have
\begin{align}
\label{main-reduction-1-1}
    \|Eg\|_{\BL^p(B_R)}^p&\lessapprox r^{3/2-p/2}R^{-p/4}\la_2^{1-p/2}\la_5^{p/2}\\ \nonumber
    & \cdot (\la_3|\cq_{\la_3,\la}^{k}|)^{1-p/2} \Big(\sum_{Q\in\cq_{\la_3,\la}^{k}}\sum_{S_t\in\cs_t(Q)}\|Eg^{\ZT_{S_t,\la_5}}\|_{L^2(Q)}^2\Big)^{p/2}.
\end{align}

\smallskip

Next, we would like to bound $\sum_{Q\in\cq_{\la_3,\la}^{k}}\sum_{S_t\in\cs_t(Q)}\|Eg^{\ZT_{{S_t},\la_5}}\|_{L^2(Q)}^2$. For a fixed $R^{1/2}$-cube  $Q\in\cq_{\la_3,\la}^{k}$, here is useful observation:  Each $r$-tube $T'$ can belong to at most $(Rr^{-1})^{1/2}\la_5^{-1}$ many $\{\ZT_{{S_t},\la_5}\}_{{S_t}\in\cs_t(Q)}$. This is because on one hand if $T'\in\ZT_{{S_t},\la_5}$, then $T'$ intersects at least $\la_5$ many $r^{1/2}$-cubes in ${S_t}\cap Q$ (equivalently, in $\bq_{S_t,Q}$). While on the other hand each $T'$ can intersect at most $(Rr^{-1})^{1/2}$ distinct $r^{1/2}$-cubes in $Q$.

This leads to 
\begin{align}
    \sum_{S_t\in\cs_t(Q)}\|Eg^{\ZT_{O,\la_5}}\|_{L^2(Q)}^2\lesssim&\sum_{S_t\in\cs_t(Q)}\sum_{\theta'\in\Theta[r]}\sum_{T'\in\ZT_{S_t,\la_5,\theta'}[r]}\int_Q|Eg_{\theta'}|^2\Id_{T'}   \\ \nonumber
    \lesssim& (Rr^{-1})^{1/2}\la_5^{-1}\sum_{\theta\in\Theta[r]}\int_Q|Eg_{\theta'}|^2.
\end{align}
Thus, after summing over $Q\in\cq_{\la_3,\la}^{k}$ we have
\begin{align}
\nonumber
    \sum_{Q\in\cq_{\la_3,\la}^{k}}\sum_{S_t\in\cs_t(Q)}\|Eg^{\ZT_{S_t,\la_5}}\|_{L^2(Q)}^2   \lessapprox&\,(Rr^{-1})^{1/2}\la_5^{-1}\\ \label{R-half-reduction-1}
    &\cdot\sum_{Q\in\cq_{\la_3,\la}^{k}}\sum_{\theta'\in\Theta[r]}\int_Q|Eg_{\theta'}|^2.
\end{align}
Inside each $Q$, invoke the $L^2$ orthogonality so that 
\begin{equation}
    \eqref{R-half-reduction-1}\lesssim\sum_{Q\in\cq_{\la_3,\la}^{k}}\sum_{T\in\ZT_{g}[R]}\int_Q|Ef_T|^2,
\end{equation}
which, via H\"older's inequality, is bounded above by 
\begin{equation}
    (|Q|\cdot|\cq_{\la_3,\la}^{k}|)^{1-2/p}\Big(\int_{\bigcup_{\cq_{\la_3,\la}^{k}}}\Big(\sum_{T\in\ZT_{g}[R]}|Ef_T|^2\Big)^{p/2}\Big)^{2/p}.
\end{equation}
Therefore, combine the above calculations with \eqref{main-reduction-1-1} so that
\begin{align}
\label{main-reduction-1-2}
     \|Eg&\|_{\BL^p(B_R)}^p\lessapprox r^{3/2-p/2}R^{-p/4}\la_2^{1-p/2}\la_3^{1-p/2}(Rr^{-1})^{p/4}R^{\frac{3(p-2)}{4}}\\ \label{incidence-esti-1}
    & \cdot\int_{\bigcup_{\cq_{\la_3,\la}^{k}}}\Big(\sum_{T\in\ZT_{g}[R]}|Ef_T|^2\Big)^{p/2}.
\end{align}

To bound \eqref{incidence-esti-1}, notice that for each $R^{-1/2}$-cap $\theta$ (see \eqref{la-1} for $\la_1$)
\begin{equation}
\label{l2-equivalence-2}
    \la_1\|Ef_{T}\|_2^2\sim\sum_{T\in\ZT_{\la_1,\theta}}\|Ef_{T}\|_2^2\lesssim R\sum_{T\in\ZT_{\la_1,\theta}}\|f_{T}\|_2^2\lesssim R\|f_{\theta}\|_2^2\lesssim \sup_{\theta}\|f_{\theta}\|_{L^2_{ave}}^2.
\end{equation}
Since we already assume above \eqref{la-1} that $\|f_T\|_2$ are all about the same, so by \eqref{l2-equivalence-2} 
\begin{equation}
    \|Ef_T\|_\infty^2\lesssim\|Ef_{T}\|_2^2/|T|\lesssim\frac{\|f_{\theta}\|_2^{4/p}\sup_{\theta}\|f_{\theta}\|_{L^2_{ave}}^{2-4/p}R^{2/p}}{R^{2}\la_1},
\end{equation}
giving the reduction 
\begin{align}
\label{kakeya-type-reduction}
   &\int_{\bigcup_{\cq_{\la_3,\la}^{k}}}\Big(\sum_{T\in\ZT_{g}[R]}|Ef_T|^2\Big)^{p/2}\\ \nonumber &\lesssim\frac{\|f_{\theta}\|_2^{2}\sup_{\theta}\|f_{\theta}\|_{L^2_{ave}}^{p-2}}{R^{p-1}\la_1^{p/2}}   \int_{\bigcup_{\cq_{\la_3,\la}^{k}}}\Big(\sum_{T\in\ZT_{g}[R]}\Id_{T}\Big)^{p/2}.
\end{align}

Finally, recall Lemma \ref{second-method-lem} that each $Q\in\cq_{\la_3,\la}^{k}$ intersects $\lessapprox R^{O(\de)} R^{1/4}\la^{-1}\la_1$ tubes $T\in\ZT_{g}[R]$ and recall \eqref{l1-estimate} (see \eqref{g} again). Hence 
\begin{equation}
    \int_{\bigcup_{\cq_{\la_3,\la}^{k}}}\Big(\sum_{T\in\ZT_{g}[R]}\Id_{T}\Big)^{p/2}\lesssim R^{11/4+p/8}\la^{2-p/2}\la_1^{p/2}(|\Theta_{\la_1}[R]|/R).
\end{equation}
Plug this back to \eqref{kakeya-type-reduction} so that
\begin{align}
\label{kakeya-type-esti-1}
    \int_{\bigcup_{\cq_{\la_3,\la}^{k}}}\Big(\sum_{T\in\ZT_{g}[R]}|Ef_T|^2\Big)^{p/2}\lesssim R^{\frac{p-2}{8}}R^{3-p}\la^{2-p/2}\|f\|_2^{2}\sup_{\theta}\|f_{\theta}\|_{L^2_{ave}}^{p-2},
\end{align}
which, together with \eqref{main-reduction-1-2}, yields our first estimate
\begin{equation}
\label{main-esti-1}
    \|Eg\|_{\BL^p(B_R)}^p\lessapprox R^{O(\de)} C_4(R,r,\vec{\la})\|f\|_2^{2}\sup_{\theta'}\|f_{\theta}\|_{L^2_{ave}}^{p-2},
\end{equation}
where the quantity $C_4(R,r,\vec{\la})$ is given as
\begin{align}
\label{C4}
    C_4(R,r,\vec{\la})=&r^{3/2-p/2}R^{-p/4}\la_2^{1-p/2} \la_3^{1-p/2}(Rr^{-1})^{p/4}R^{\frac{3(p-2)}{4}}R^{\frac{p-2}{8}}R^{3-p}\la^\frac{4-p}{2}. 
\end{align}

\smallskip

The first estimate \eqref{main-esti-1} works well when $\la_3$ is large (recall \eqref{la-3-t} that $\la_3$ denotes the number of fat surfaces $S_t$ associated to an $R^{1/2}$-ball). While if $\la_3$ is small, we need another estimate.

\subsection{Case two} Our second estimate uses the refined decoupling in \cite{GIOW} (see \cite{Bourgain-Demeter-l2} for the original Bourgain-Demeter decoupling theorem). 
\begin{theorem}[\cite{GIOW} Theorem 4.2]
\label{refined-decoupling-thm}
Let $2\leq p\leq 4$. Suppose $h$ is a sum of scale-$R$ wave packets $h=\sum_{T\in\ZW}f_T$ so that $\|Ef_T\|_{L^2(B_R)}^2$ are about the same up to a constant multiple. Let $Y$ be a union of $R^{1/2}$-balls in $B_R$ such that each $R^{1/2}$-ball $Q\subset Y$ intersects to at most $M$ tubes from $T\in\ZW$. Then there exists $\be\ll\e'\ll\e$ (for example, $\e'=\be^{1/2}\ll\e_0$) so that
\begin{equation}
\label{refined-decoupling}
    \|Eh\|_{L^p(Y)}\leq R^{\e'}\Big(\frac{M}{|\ZW|}\Big)^{\frac{1}{2}-\frac{1}{p}}\Big(\sum_{T\in\ZW}\|Ef_T\|_p^2\Big)^{1/2 }.
\end{equation}
\end{theorem}
We know from Lemma \ref{second-method-lem} item (8) and \eqref{non-related-BLp}
\begin{align}
    \|Eg\|_{\BL^p(B_R)}^p\lesssim R^{O(\de)} \sum_{Q\in\bar\cq_{\la_3,\la}^{k}}\|Eg\|_{L^p(Q)}^p.
\end{align}
Also, again from Lemma \ref{second-method-lem} item (8) we know that for each $Q\in\bar\cq_{\la_3,\la}^{k}$, there are at most $\sim\mu\lessapprox R^{O(\de)} R^{1/4}\la^{-1}\la_1$ many $R$-tubes in $\ZT_g[R]$ that intersect it. Hence by the refined decoupling \eqref{refined-decoupling} one has (recall \eqref{la-1} for the definition of $\la_1$) 
\begin{equation}
    \sum_{Q\in\bar\cq_{\la_3,\la}^{k}}\|Eg\|_{L^p(Q)}^p\lesssim R^{O(\de)} (\mu\la_1^{-1}|\Theta_{\la_1}[R]|^{-1})^{p/2-1}\Big(\sum_T\|Ef_T\|_p^2\Big)^{p/2}.
\end{equation}
By Bernstein's inequality and Plancherel,
\begin{equation}
    \Big(\sum_T\|Ef_T\|_p^2\Big)^{p/2}\lesssim R^{2-p}\Big(\sum_T\|Ef_T\|_2^2\Big)^{p/2}\lesssim R^{2-p/2}\Big(\sum_T\|f_T\|_2^2\Big)^{p/2},
\end{equation}
which is bounded above by
\begin{equation}
    R^{3-p}|\Theta_{\la_1}'[R]|^{(p-2)/2}\|f\|_2^2\sup_{\theta'}\|f_{\theta'}\|_{L^2_{ave}}^{p-2}.
\end{equation}

Thus, the above calculations give
\begin{equation}
\label{main-esti-2}
    \|Eg\|_{\BL^p(B_R)}^p\lessapprox R^{O(\de)} C_5(R,r,\vec{\la})\|f\|_2^{2}\sup_{\theta'}\|f_{\theta'}\|_{L^2_{ave}}^{p-2},
\end{equation}
where the quantity $C_5(R,r,\vec{\la})$ has the expression
\begin{equation}
    C_5(R,r,\vec{\la})=(\mu\la_1^{-1})^{p/2-1}R^{3-p}.
\end{equation}

\subsection{Case three}

Our third and final estimate use the information on $\bq_{S_t}$ (see Lemma \ref{second-method-lem}, in particular, $|\bq_{S_t}|\lesssim\la_2\la_6$). On one hand, for each fat surface $S_t$, by Lemma \ref{polynomial-plane-lemma},  Lemma \ref{planar-refinement-lem} (only \eqref{planar-refinement-1}), and the fact that broad-norm is essentially dominated by bilinear norm, one has
\begin{equation}
    \sum_{q\in\bq_{S_t}}\|Eg^{\ZT_{S_t}}\|_{\BL^p(q)}^p\lesssim (\la_2\la_6)^{1-p/4}r^{\frac{3-p}{2}}r^{-p/2}\|Eg^{\ZT_{S_t}}\|_{L^2(B_{S_t})}^p.
\end{equation} 
Here we use the estimate $|\bq_{S_t}|\lesssim\la_2\la_6$, which is given in the line below \eqref{la-6-t}. Note that by definition the function $g^{\ZT_{S_t}}$ is a sum of all tangential wave packets intersecting $S_t$, so we cannot use \eqref{polynomial-l2} to sum up $\|g^{\ZT_{S_t}}\|_2^2$ for all fat surfaces $S_t$.

On the other hand, we similarly have (see \eqref{second-method-beginning})
\begin{equation}
    \sum_{q\in\bq_{S_t}}\|Eg^{\ZT_{S_t}}\|_{\BL^p(q)}^p\sim\sum_{q\in\bq_{S_t}}\|Eg_{S_t}\|_{\BL^p(q)}^p\lesssim (\la_2\la_6)^{1-p/4}r^{\frac{3-p}{2}}\|g_{S_t}\|_{2}^p.
\end{equation}
Recall \eqref{second-method-beginning}. Summing up all $S_t\in\cs_t$ and by the broom estimate \eqref{broom-esti-2} (see also Remark \ref{broom-remark-0}, and notice $\ka_2(j)\geq1$) we have similar to \eqref{case3-2} that
\begin{equation}
\label{main-esti-3}
    \|Eg\|_{\BL^p(B_R)}^p\lessapprox R^{O(\e_0)} D(\la_2\la_6)^{1-p/4}r^{\frac{3-p}{2}} R^{-\frac{p-2}{4}}\|f\|_2^2\cdot\sup_{\theta}\|f_\theta\|_{L^2_{ave}(\theta)}^{p-2}.
\end{equation}
Also, by pigeonholing on $S_t\in\cs_t$, \eqref{polynomial-l2}, and \eqref{second-method-beginning}, we have similar to \eqref{case3-1} that
\begin{equation}
\label{main-esti-4}
    \|Eg\|_{\BL^p(B_R)}^p\lessapprox R^{O(\de)} D^{3-p}(\la_2\la_6)^{1-p/4}r^{\frac{3-p}{2}}\|f\|_2^p.
\end{equation}

\subsection{Numerology} 

Finally, combining and simplifying \eqref{main-esti-1}, \eqref{main-esti-2}, \eqref{main-esti-3}, \eqref{main-esti-4}, one hopes to find the smallest $p$ such that the minimum the following equations 
\begin{equation}
\label{system-0}
    \begin{cases}
    \Big(\dfrac{R}{r}\Big)^{\frac{3p-6}{2}}\la_2^\frac{2-p}{2}\la_3^{\frac{2-p}{2}}R^{3-p}R^\frac{p-2}{8}\la^\frac{4-p}{2}\\[2ex]
    (\mu\la_1^{-1})^{\frac{p-2}{2}}R^{3-p}\\[2ex]
    D^{3-p}r^{\frac{3-p}{2}}\la_2^\frac{4-p}{4}\la_6^\frac{4-p}{4}\\[2ex]
    DR^{-\frac{p-2}{4}}r^{\frac{3-p}{2}}\la_2^\frac{4-p}{4}\la_6^\frac{4-p}{4}
    \end{cases}
\end{equation}
is bounded above by $1$. Recall also the relations
\begin{enumerate}
    \item $\la_6|\cs_t|\lessapprox R^{O(\de)} \la_3|\cq_{\la_3}^{k}|\lesssim\la_6|\cs_t|$ (see Lemma \ref{second-method-lem}).
    \item $|\cs_t|\gtrapprox R^{-O(\de)}
    D^3$ (see Lemma \ref{second-method-lem} and \eqref{cs-t}).
    \item $|\cq_{\la_3}^{k}|\lessapprox|\cq_{\la_3,\la}^{k}|\lessapprox R^{O(\de)} \la\la_1R^{3/2}\mu^{-1}$ and $\mu\lessapprox R^{O(\de)}\la^{-1}\la_1R^{1/4}$ (see Lemma \ref{second-method-lem} item (8)), yielding $\la_6D^3\la_3^{-1}\lessapprox R^{O(\de)}\la\la_1R^{3/2}\mu^{-1}$.
\end{enumerate}

Use the estimate $\mu\lessapprox R^{O(\de)}\la^{-1}\la_1R^{1/4}$ to rewrite the first term of system \eqref{system-0}:
\begin{equation}
\label{system-1}
    \begin{cases}
    \Big(\dfrac{R}{r}\Big)^{\frac{3p-6}{4}}\la_2^\frac{2-p}{2}\la_3^{\frac{2-p}{2}}\la^\frac{2-p}{2}R^{3-p}R^\frac{p}{8}\mu^{-1}\la_1\\[2ex]
    (\mu\la_1^{-1})^{\frac{p-2}{2}}R^{3-p}\\[2ex]
    D^{3-p}r^{\frac{3-p}{2}}\la_2^\frac{4-p}{4}\la_6^\frac{4-p}{4}\\[2ex]
    DR^{-\frac{p-2}{4}}r^{\frac{3-p}{2}}\la_2^\frac{4-p}{4}\la_6^\frac{4-p}{4}
    \end{cases}
\end{equation}
Combining the first two equations to get rid of $\la_3, \la, \mu, \la_1$ (multiply the first one with the $\frac{p}{p-2}$ power of the second one, then use $\la_6D^3\la_3^{-1}\lessapprox R^{O(\de)}\la\la_1R^{3/2}\mu^{-1}$), one gets 
\begin{equation}
    \begin{cases}
    I=\Big(\dfrac{R}{r}\Big)^{\frac{3p-6}{4}}\la_2^\frac{2-p}{2}\la_6^{\frac{2-p}{2}}D^{\frac{6-3p}{2}}R^{\frac{-9p^2+38p-24}{8p-16}}\\[2ex]
    II=D^{3-p}r^{\frac{3-p}{2}}\la_2^\frac{4-p}{4}\la_6^\frac{4-p}{4}\\[2ex]
    III=DR^{-\frac{p-2}{4}}r^{\frac{3-p}{2}}\la_2^\frac{4-p}{4}\la_6^\frac{4-p}{4}
    \end{cases}
\end{equation}
Calculate $I\times (III^\frac{p}{4(p-2)}II^{1-\frac{p}{4(p-2)}})^{\frac{2(p-2)}{4-p}}$ to obtain
\begin{equation}
    R^\frac{p^3-11p^2+26p}{-4p^2+24p-32}r^\frac{-p^2+2p}{-4p+16}
\end{equation}
One gets that when $\log r/\log R\geq 2/3$ the above is bounded by 1 if $p>3.213$. Combining the result in \eqref{result-1} we have
\begin{equation}
    \|Eg\|_{\BL^p(B_R)}^p\lesssim R^{O(\de)}\|f\|_2^{2}\sup_{\theta}\|f_{\theta}\|_{L^2_{ave}}^{p-2}
\end{equation}
if $p>45/14$. This proves our main estimate \ref{main-esti}.

\begin{remark}
\rm
The first method in Section \ref{first-method} does not work very well when $r$ is much bigger than $R^{2/3}$. It mainly is because in this case, we have to focus on $r$-balls and $r\times r\times R$ tubes instead of $R/r^{1/2}$-balls and $R/r^{1/2}\times R/r^{1/2}\times R$-tubes. A (bad) consequence is that we cannot get \eqref{wolff-kakeya-1}, \eqref{estimate-c} from \eqref{estimate-a}, \eqref{estimate-b}, since each $r\times r\times R$ tube may contains a lot of $R/r^{1/2}\times R/r^{1/2}\times R$-tube, each of which is a target fat tube in the broom estimate \eqref{broom-esti-1}.

While for $r$ near $R^{2/3}$ our first method still shows some strength. In particular, when $r=R^{2/3}$ one can optimize the system \eqref{system-4} to obtain the range $p>3.2$. On the other hand, our second method gives a better range of $p$ when $r$ is larger. If we optimize these two methods for $r\geq R^{2/3}$, we can indeed have
\begin{equation}
    \|Eg\|_{\BL^p(B_R)}^p\lesssim R^{O(\de)}\|f\|_2^{2}\sup_{\theta}\|f_{\theta}\|_{L^2_{ave}}^{p-2}
\end{equation}
when $p>3.21$ and $r\geq R^{2/3}$.

\end{remark}

\bibliographystyle{alpha}
\bibliography{bibli}

\end{document}